\documentclass[reqno]{amsart}

\usepackage{verbatim}

\usepackage{mathtools}
\usepackage{mathrsfs}
\usepackage{amssymb}
\usepackage{soul,xcolor}
\usepackage{stackrel}
\usepackage{adjustbox}

\setstcolor{red}

\usepackage{tikz-cd}
\tikzcdset{
  cells={font=\everymath\expandafter{\the\everymath\displaystyle}},
}

\usepackage{latexsym,amssymb,amsmath}
\usepackage{epsfig}
\input xy
\xyoption{all}
\usepackage[english]{babel}

\usepackage{amsmath}
\usepackage{bbm} %for \mathbbm{1}
\newtheorem{theorem}{Theorem}[section]
\newtheorem{lemma}[theorem]{Lemma}
\newtheorem{prop}[theorem]{Proposition}
\newtheorem{proposition}[theorem]{Proposition}

\newtheorem{corollary}[theorem]{Corollary}
\newtheorem{definition}[theorem]{Definition}

\newtheorem{notation}[theorem]{Notation}
\newtheorem{facts}[theorem]{Facts}
\newtheorem{fact}[theorem]{Fact}
\newtheorem{setup}[theorem]{Setup}
\newtheorem{question}[theorem]{Question}

\theoremstyle{definition}
\newtheorem{example}[theorem]{Example}
\newtheorem{remarks}[theorem]{Remarks}
\newtheorem{remark}[theorem]{Remark}

\newcommand{\holom}[3]{\ensuremath{#1:#2  \rightarrow #3}}

\def\PP{\mathbb{P}}
\def\CC{\mathbb{C}}
\def\ZZ{\mathbb{Z}}
 
\def\cA{{\mathcal A}}

\def\cF{{\mathcal F}}

\def\cM{{\mathcal M}}

\def\cU{{\mathcal U}}
\def\cV{{\mathcal V}}
\def\cW{{\mathcal W}}

\def\fg{{\mathfrak g}}
\def\fso{{\mathfrak{so}}}\def\fsl{{\mathfrak{sl}}}

\def\fg{{\mathfrak g}}

\DeclareMathOperator{\rs}{{rs}}

\DeclareMathOperator{\Aut}{Aut}

\DeclareMathOperator{\End}{End}
\DeclareMathOperator{\SL}{SL}
\DeclareMathOperator{\GL}{GL}

\DeclareMathOperator{\SO}{SO}

\DeclareMathOperator{\Spin}{Spin}
\DeclareMathOperator{\Clif}{Clif}
\DeclareMathOperator{\Pin}{Pin}

\DeclareMathOperator{\id}{id}

\DeclareMathOperator{\diag}{diag}

\DeclareMathOperator{\Nm}{Nm}

\DeclareMathOperator{\SC}{SC}

\usepackage{xcolor}

\def\and#1{\textcolor{red}{{\bf *** And:} #1 {\bf ***}}}

\newcommand{\Z}{\ensuremath{\mathbb{Z}}}
\newcommand{\C}{\ensuremath{\mathbb{C}}}

\newcommand\sO{{\mathscr O}}

\DeclareMathOperator{\MG}{\mathcal M_G}

\usepackage{hyperref}

\author{Vladimiro Benedetti}
\author{Andreas H\"oring}
\author{Jie Liu}

\address{Vladimiro Benedetti, Universit\'e C\^ote d'Azur, CNRS, LJAD, France}
\email{vladimiro.benedetti@univ-cotedazur.fr}

\address{Andreas H\"oring, Universit\'e C\^ote d'Azur, CNRS, LJAD, France}
\email{Andreas.Hoering@univ-cotedazur.fr}

\address{Jie Liu, Institute of Mathematics, Academy of Mathematics and Systems Science, Chinese Academy of Sciences, Beijing, 100190, China}
\email{jliu@amss.ac.cn}

\title[Hitchin morphism for intersections of two quadrics]{Intersection of two quadrics: modular interpretation and Hitchin morphism}
\date{June 5th, 2025}


\begin{document}

\begin{abstract}
The cotangent bundle $T^*X$ of a smooth intersection $X$ of two quadrics admits a Lagrangian fibration determined by the intrinsic geometry of $X$. We show that this fibration is actually the Hitchin morphism if we endow $X$ with a structure of moduli space of twisted $\Spin$-bundles. This generalises the classical result for threefolds, in which case it recovers the Hitchin fibration for the moduli space of rank two bundles with fixed determinant of odd degree on a curve of genus two.
\end{abstract}

\maketitle 

%\tableofcontents

\section{Introduction}

A classical result of  P.~Newstead \cite[Theorem 1]{Newstead1968} says that a smooth complete intersection $X\subset \PP^5$ of two quadrics is isomorphic to a moduli space $\cU_C(2,L)$, where $C$ is a smooth projective curve of genus $2$ and $L$ is a line bundle of odd degree over $C$ (see also \cite{DesaleRamanan1976,Bhosle1984,LiZhang2024}). In particular, the Hitchin morphism then yields a Lagrangian fibration $T^*X\rightarrow \CC^3$. 

The same phenomenon persists for higher-dimensional smooth complete intersections of two quadrics, as established in \cite{BeauvilleEtesseHoeringLiuVoisin2024} without recourse to the Hitchin morphism, relying solely on the intrinsic geometry of $X$. Note that the algebra of symmetric tensors $H^0(X,S^{\bullet} T_X)$ of a projective manifold $X$ can be canonically identified to the ring of regular functions on $T^*X$. We recall:

 \begin{theorem}[\protect{\cite[Theorem 1.1]{BeauvilleEtesseHoeringLiuVoisin2024}}]
 	\label{thm.BEHLV}
    Let $X\subset \PP^{n+2}$ be a smooth complete intersection of two quadrics, $n\geq 2$.
     The natural morphism 
     \[
     \Phi_X\colon T^*X \longrightarrow (H^0(X,S^2 T_X))^*\cong \CC^{n}
     \]
     is a Lagrangian fibration.
 \end{theorem}

If $X$ is a surface \cite{BiswasHollaKumar2010,Casagrande2015,KL22} or threefold \cite[Sect.2]{BeauvilleEtesseHoeringLiuVoisin2024} it has a well-known structure of moduli space\footnote{We recommend \cite{DPS24} and \cite{Hit24} for a thorough investigation involving the threefold in $\PP^5$.} proving a modular interpretation of the Lagrangian fibration .
Thus it is natural to ask:
\begin{question}
	\label{q.QBEHL}
	Does $X$ admit a structure of modular space of principal bundles
    for which $\Phi_X$ is the Hitchin morphism?
\end{question}

In this paper we use the structure of moduli space of twisted Spin-bundles
for the case of odd-dimensional complete intersections \cite{Ram1981}
to give a complete answer to Question \ref{q.QBEHL} in every dimension.

\subsection{General background}

Let $C$ be a smooth projective curve of genus $g\geq 2$. Let $\cU_C(r,L)$ be the moduli space of semistable bundles of rank $r$ with a fixed determinant $L$ of degree $d$ over $C$ such that $(r,d)=1$. Let $[F]\in \cU_C(r,L)$ be a stable vector bundle over $C$. Then the cotangent space $T^*_{\cU_C(r,L),[F]}$ can be naturally identified to $H^0(C,\End_0(F)\otimes K_C)$, which consists of traceless $K_C$-valued endomorphisms of $F$. Let 
\[
\cA\coloneqq H^0(C,K_C^2)\oplus \dots \oplus H^0(C,K_C^r)
\]
be the \emph{Hitchin base},
then  the \emph{Hitchin morphism} is defined as
\[
h_{\cU_C(r,L)}\colon T^*\cU_C(r,L)\rightarrow \cA \qquad
([F],\theta) \longmapsto (s_2(\theta),\dots,s_r(\theta)),
\]
where $\theta\in H^0(C,\End_0(F)\otimes K_C)$ and $s_j(\theta)$'s are the coefficients of the characteristic polynomial of $\theta$, i.e. $s_j(\theta) \coloneqq (-1)^j \operatorname{tr}(\wedge^j \theta)$. 
N.~Hitchin proved in \cite{Hit87} that $h_{\cU_C(r,L)}$ is actually a Lagrangian fibration for the natural symplectic form on the cotangent bundle.

While a complete intersection $X$ of two quadrics may not be contained
in a moduli space of $\mbox{GL}_n$-bundles, Ramanan \cite{Ram1981} observed that for $X$ of odd dimension there is an embedding
into a moduli space of twisted $\mbox{Spin}$-bundles. We will use this modular structure  to answer Question \ref{q.QBEHL} in odd dimension, in a second
step we show how to extend the construction in even dimension.

\subsection{Odd dimensional case} \label{intro-odd-dimensional}

 Firstly we consider the odd dimensional case, i.e. $n=2g-1$ for some $g\geq 2$. In suitable coordinates, we can assume that $X$ is given by the following two equations:
 \[
 q_1\coloneqq \sum_{j=0}^{2g+1} x_j^2 \quad \text{and}\quad q_2\coloneqq \sum_{j=0}^{2g+1} \lambda_j x_j^2.
 \]
We can associate to $X$ a hyperelliptic curve $\pi\colon C\rightarrow \PP^1$ of genus $g$, see Setup \ref{setup-odd}, and denote by $i$ the covering involution. Denote by $p_0,\dots,p_{2g+1}\in C$ the Weiertra\ss\ points and denote by $\Delta\subset \PP^1$ the branch locus with $r_j\coloneqq \pi(p_j)$. Let $h$ be the pull-back $\pi^*\sO_{\PP^1}(1)$ and let $\alpha\coloneq h^{g-1}\otimes \sO_{C}(p_{2g+1})$. 

Let $\cM$ be the moduli space of semistable twisted $\Spin_{2g}$-bundles over $C$ such that the associated orthogonal bundles are $i$-invariant with fixed type $\tau=(1^{2g+1},2g-1)$ (see Notation \ref{notation:MandU} for the precise definition of $\cM$). S.~Ramanan proved in \cite[Theorem 3]{Ram1981} (see also \S\,\ref{section:twoquadrics}) that there exists an isomorphism (determined by $\alpha$)
\begin{equation}
	\label{eq.Isom-X-M}
	X\stackrel{\cong}{\longrightarrow} \cM
\end{equation}
such that the natural involution $i\colon \cM\rightarrow \cM$, $[\cF]\mapsto[i^*\cF]$, corresponds to the involution of $X$ by changing the sign of the coordinate $x_{2g+1}$.

The moduli space $\cM$ admits a natural forgetful map to the moduli space
of $i$-invariant orthogonal bundles, a key step is to use this structure to describe the cotangent space:

 \begin{prop}[\protect{= Proposition \ref{prop.Cotangent-cM}}]
 	\label{prop:intro-cotangent-cM}
 	Let $[\cF]\in \cM$ be a point with the associated orthogonal vector bundle $(F,q)$. Then we have a canonical isomorphism
 	\[
 	T^*_{\cM,[\cF]} \cong H^0(C,\wedge^2 F\otimes K_C)^+.
 	\]
    where the superscript indicates the $+1$-eigenspace with respect to the $i$-action.  
 \end{prop}
 
The Hitchin morphism for $\cM$ is therefore defined as follows (cf. Remark \ref{remark-natural}):
 \[
 \begin{tikzcd}[row sep=tiny]
 	 h_{\cM}\colon T^*\cM  \arrow[r]
 	     & \cA_{\Spin_{2g}}\coloneqq \bigoplus_{j=1}^{g-1} H^0(C,K_C^{2j}) \oplus H^0(C,K_C^g) \\
 	([\cF],\theta)\arrow[r,mapsto]
 	     & (\operatorname{tr}(\wedge^2\theta),\dots,\operatorname{tr}(\wedge^{2g-2}\theta),\operatorname{Pf}(\theta))
 \end{tikzcd}
 \]
 where $\theta\in H^0(C,\wedge^2 F\otimes K_C)^+$ is viewed as an element in $T^*_{\cM,[\cF]}$ under the isomorphism in Proposition \ref{prop:intro-cotangent-cM} and $\operatorname{Pf}(\theta)$ is the Pfaffian of $\theta$.

 \begin{theorem}
 	\label{mainthm.Odd-dimension}
        Let $X\subset \PP^{2g+1}$ be a smooth complete intersection of two quadrics and let $\cM$ be the corresponding moduli space in \eqref{eq.Isom-X-M}. Then the following holds: 
 	\begin{enumerate}
 		\item The image of $h_{\cM}$ is equal to 
        \[
        H^0(C,K_C^2)^+ \cong H^0(\PP^1, K_{\PP^1}^2\otimes \sO_{\PP^1}(\Delta)) \cong \CC^{2g-1}.
        \]
 		
 		\item The morphism $h_{\cM}$ coincides with the map $\Phi_X$ from Theorem \ref{thm.BEHLV} under the isomorphism \eqref{eq.Isom-X-M}.
 	\end{enumerate}
 \end{theorem}

 \subsection{Even dimensional case}

Next we consider the even dimensional case $n=2g-2$, i.e. a smooth complete intersection of two quadrics $Y\subset \PP^{2g}$. Again in suitable coordinates, we can assume that $Y$ defined by the following two equations:
\[
q_1\coloneqq \sum_{j=0}^{2g} x_i^2 \quad \text{and}\quad q_2\coloneqq \sum_{j=0}^{2g} \lambda_j x_j^2,
\]
Then $Y$ can be naturally embedded into an odd dimensional smooth complete intersection of two quadrics $X\subset \PP^{2g+1}$ as in the previous subsection such that $Y=X\cap \{x_{2g+1}=0\}$. Let $\cM$ be the moduli space corresponding to $X$ in \eqref{eq.Isom-X-M} equipped the natural involution $i\colon \cM\rightarrow \cM$.

\begin{prop}[\protect{= Proposition \ref{prop.Cotangent-cMi}}]
	The variety $Y$ is identified to the $i$-fixed locus $\cM^i$ under the isomorphism \eqref{eq.Isom-X-M}. Moreover, for any point $[\cF]\in \cM^i$ with $(F,q)$ the associated orthogonal bundle, we have a canonical isomorphism
	\[
	T^*_{\cM^i,[\cF]} \cong H^0(C,\wedge^2 F\otimes K_C\otimes \sO_C(-p_{2g+1}))^+.
	\]
\end{prop}

The Hitchin morphism $h_{\cM}$ can then be naturally restricted to $T^*\cM^i$, with the resulting morphism denoted by $h_{\cM^i}$, and we obtain the analogoue of Theorem \ref{mainthm.Odd-dimension} in the even dimensional case.

\begin{theorem}
	\label{mainthm.Evendim}
      Let $Y\subset \PP^{2g}$ be a smooth complete intersection of two quadrics and let $\cM^i$ be the corresponding moduli space.
	The map $h_{\cM^i}$ coincides with the map $\Phi_Y$ under the isomorphism $Y\cong \cM^i$ induced by \eqref{eq.Isom-X-M}.
\end{theorem}

 \begin{remark}
     In a  recent work  Hitchin studies the integrable system defined by the map $\Phi_X\colon T^*X  \to \CC^n$ of Theorem \ref{thm.BEHLV}. He shows \cite[Prop.1]{HitLaum} that the locus of points 
     $$\cW\coloneq\{x\in X \mid \Phi_X^{-1}(0)\cap T_{x}^* X\not=0\}\subset X$$ is a  hypersurface in $X$. Using Theorems \ref{mainthm.Odd-dimension} and \ref{mainthm.Evendim}, we can reinterpret the locus $\cW$ as the \emph{wobbly} locus, i.e. the complement of the locus of \emph{very stable} bundles.
 \end{remark}

\subsection*{Acknowledgements} 

J.~Liu would like to thank Yanjie Li for bringing his attention to \cite{Ram1981}, which is the starting point of this paper.  
We thank H. Zelaci for discussions around his inspiring paper \cite{Zelaci2022}.

A.H. is partially supported by the projects ANR-23-CE40-0026 ``POK0'',
ANR-24-CE40-3526 ``GAG'',
and the France 2030 investment plan managed by the National Research Agency (ANR), as part of the Initiative of Excellence of Universit\'e C\^ote d'Azur under reference number ANR-15-IDEX-01.

J.L. is supported by the National Key Research and Development Program of China (No. 2021YFA1002300), the CAS Project for Young Scientists in Basic Research (No. YSBR-033), the NSFC grant (No. 12288201) and and the Youth Innovation Promotion Association CAS.

\section{Principal bundles and Hitchin morphism}

We work over the complex numbers, for general definitions we refer to \cite{Har77}. 
Varieties will always be supposed to be irreducible and reduced. All algebraic groups are affine. All the algebraic groups appearing this paper are with coefficients in $\C$,
and we will abbreviate $\GL_n \CC, \mbox{SO}_n \C, \ldots$ by $\GL_n, \mbox{SO}_n, \ldots$. 

Let $C$ be a smooth projective curve of genus $g\geq 2$, and let $G$ be a reductive algebraic group. 

\subsection{Moduli spaces of principal bundles}
\label{ss.Moduli-Hitchin-map}

We briefly recall the definition and basic properties of the moduli spaces of principal $G$-bundles over $C$ --- see \cite{Ramanathan1996,Ramanathan1996a}.

\begin{definition}
A $G$-principal bundle $\cF$ is a variety with a right $G$-action admitting 
a $G$-invariant fibration $p\colon \cF \to C$ which is locally trivial in the Zariski topology.
\end{definition}

Let $V$ be a be a representation of the algebraic group $G$, we define the variety
$$
F \coloneqq \cF_V \coloneqq\cF \times^G V
$$
as the quotient of $\cF \times V$ by the $G$-action $g \cdot (f,v)=(fg,g^{-1}v)$.
The $G$-action descends to $\cF_V$ and $\cF_V \rightarrow C$ defines a vector bundle on $C$.  As a special case we denote by $\cF_\fg \rightarrow C$ the vector bundle given by the adjoint action of $G$ on its Lie algebra $\fg$.

The same construction can be applied to any variety $V$ with $G$-action. Moreover,
given a $G$-equivariant morphism $f:V \to V'$ of $G$-varieties, one also obtains the corresponding morphism $\cF_f\colon \cF_V \to \cF_{V'}$.

\begin{example}{\rm
When $G$ is $\GL_n$ and $V$ is the standard representation of $\GL_n$, the above construction can be reversed. Indeed $\cF$ can be reconstructed from $\cF_V$ as its frame bundle. This gives a bijection between $\GL_n$-principal bundles and vector bundles of rank $n$.}
\end{example}

We refer the reader to \cite[Definition 2.13]{Ramanathan1996} for the notion of (semi-)stability for general principal $G$-bundles. In particular, $\cF$ is semistable if and only if $\cF_{\fg}$ is semistable as a vector bundle (\cite[Corollary 3.18]{Ramanathan1996}). Moreover, $\cF$ is called \emph{regularly stable} if $\cF$ is stable and $\Aut(\cF)$ is equal to the center $Z(G)$ of $G$.

\begin{notation} \label{notation:MG}
Let $G$ be a connected reductive algebraic group.
We denote by 
$$
\MG \coloneqq  \mathcal M_{G,C}
$$ the (coarse) moduli space of (equivalence classes of) semistable $G$-principal bundles
as defined in \cite{Ramanathan1996}.

For simplicity of notation we will denote by $[\cF] \in \MG$ the point corresponding
to a semistable $G$-principal bundle $\cF \rightarrow C$.
\end{notation}

By \cite[Theorem 5.9]{Ramanathan1996a}, the connected components $\cM_G^{\gamma}$ of $\cM_G$ are irreducible normal projective varieties parametrized by an element $\gamma\in \pi_1(G)$ with
\[
\dim \cM_G^{\gamma} = (g-1) \dim G + \dim Z(G).
\]

The open subset $\cM_{G}^{\rs}$ of $\cM_G$ corresponding to regularly stable $G$-bundles is smooth, and its complement in $\cM_G$ is of codimension $\geq 2$ except when $C$ is of genus $2$ and $G$ maps onto $\operatorname{PGL}_2$  \cite[Theorem II.6]{Faltings1993}. 

\subsection{Hitchin morphism}

Fix an isomorphism  $\fg\cong \fg^\vee$ as $G$-representations
(e.g. if $G$ is semisimple). Let $\cF$ be a regularly stable $G$-bundle. Then $\cM_G$ is smooth at $[\cF]$ and the tangent space identifies to $H^1(C,\cF_\fg)$. 
Using Serre duality and the fact that $\fg \cong \fg^\vee$  one obtains
\begin{equation}
\label{tangentMG}
T_{\MG,[\cF]} = 
H^1(C,\cF_\fg) =
H^0(C,\cF_\fg \otimes K_C)^*.
\end{equation}

\begin{example}
The Lie algebra of $\GL_n$ is $\mathfrak{gl}_n= \End(V)$, where $V$ is again the standard representation of $\GL_n$. As a $\GL_n$-representation, $\End(V)$ decomposes in the direct sum of $\End_0(V)$ (traceless matrices) and $\CC \id$ (trivial representation). By functoriality we get $\cF_{\fg}=\End(F)=\End_0(F) \oplus \sO_C$. 

The geometric meaning of this decomposition is as follows:
Given a stable vector bundle $F \rightarrow C$, let $\det(F)$ be its determinant. Let us denote by $\cU_C(n,d)$ the moduli space of semistable vector bundles of rank $n$ and determinant of degree $d$, so $\cU_C(n,d) \subset \mathcal M_{\GL_n}$ is an irreducible normal projective variety. Then we have a morphism
$$
\cU_C(n,d) \longrightarrow J^d(C), \qquad F \ \longmapsto \ \det(F)
$$
where $J^d(C)$ is the Jacobian of degree $d$ line bundles on $C$.
We denote by $\cU_C(n, L)$ the fibre over the line bundle $[L] \in J^d(C)$.
The fibre over the trivial bundle $[\sO_C] \in J^0(F)$ is isomorphic to 
$\mathcal M_{\SL_n}$, and the differential at $[F]$ is the projection 
$$
T_{\cU_C(n,d),[F]} = H^1(C,\End(F)) \to H^1(C,\sO_C) =T_{J^d(C),[\det(F)]}.
$$ 
Thus we deduce that 
$$
T_{\cM_{\SL_n},[F]} =H^1(C,\End_0(F)).
$$ 
Notice that $\End_0(V) = \fsl_n$ is the Lie algebra of $\SL_n $.
\end{example}

As we have mentioned, the construction $V \mapsto \cF_V$ is functorial. Let us fix a homogeneous $G$-invariant polynomial $f \in \CC[\fg^*]$ of degree $\nu$, i.e. a homogeneous $G$-equivariant map $f:\fg \to \CC$. 
Applying functoriality to $f \otimes \id_{K_C}$
we get thus a  map 
$$
\cF_f: \cF_\fg \otimes K_C \to \sO_C \otimes K_C^\nu=K_C^\nu,
$$
and a  morphism at the level of sections: $H^0(C,\cF_\fg \otimes K_C) \to H^0(C,K_C^\nu)$. 

When $G$ is semisimple, by a well-known theorem by Chevalley, the algebra of $G$-invariants in $\CC[\fg^*]$ is a polynomial algebra $\CC[f_1,\cdots,f_r]$, where $r$ is the rank of $G$ and $f_i$ are homogeneous of degree $\nu_i$ for $1 \leq i \leq r$. 

\begin{definition}[\protect{\cite[\S\,4]{Hit87}}] \label{defn:hitchin-general} 
Assume that $G$ is semisimple. Given a $G$-principal bundle $\cF$  on $C$, we set
$$
h_\cF \coloneqq  \bigoplus_{i=1}^r \cF_{f_i}\colon  \ H^0(C,\cF_\fg \otimes K_C) \longrightarrow \bigoplus_{i=1}^r H^0(C,K_C^{\nu_i}) \eqqcolon \cA_G
$$ 
Using the identification \eqref{tangentMG} we obtain a map
$$
h_G\colon T^* \cM_G^{\rs} \longrightarrow  \cA_{G}
$$
called the Hitchin morphism, and $\cA_G$ is called the Hitchin base.
\end{definition}

\begin{remark} \label{remark:killing}
We will denote by $h_i$ the restriction of the Hitchin morphism coming from the $i$-th invariant $f_i$. Recall that for any Lie algebra there exists a \emph{natural} degree $2$ invariant form $f_1$, which is the Killing form (\cite{FH91}). We thus get a morphism
$$
h_1: T^* \cM_G^{\rs} \longrightarrow H^0(C,K_C^{2}).
$$
Notice that, since the Killing form $f_1$ has degree $2$, the map $h_1$ also has degree $2$.
\end{remark}

\begin{example}
	\label{ex.Charac-poly-SO}
	For classical groups, a basis for the invariant polynomials can be described by the coefficients of the characteristic polynomial (\cite[\S\,5]{Hit87})). In particular, for $G=\SO_{2n}$, the distinct eigenvalues of a matrix $A\in \mathfrak{so}_{2n}$ occur in pairs $\pm \zeta_i$, and thus the characteristic polynomial of $A$ is of the form
	\[
	\det(x\operatorname{Id}-A) = x^{2n} + a_1 x^{2n-2} + \dots + a_{n-1} x^2 + a_n.
	\]
	In this case the coefficient $a_n$ is the square of a polynomial $\operatorname{Pf}$, the \emph{Pfaffian}, of degree $n$. Then a basis for the invariant polynomials on the Lie algebra $\mathfrak{so}_{2n}$ is 
	\[
	a_1,a_2,\dots,a_{n-1},\operatorname{Pf}.
	\]
\end{example}

\section{Moduli spaces of Spin and orthogonal bundles}

\subsection{Spin, Clifford and orthogonal groups}
\label{subsectionclifford}

Fix $m\geq 4$ an even positive integer. Let $M$ be a complex vector space with the standard quadratic form $q_m$. 
We collect some basic facts about Spin, Clifford and orthogonal groups. The \emph{orthogonal group} is defined as the group 
$$
\operatorname{O}_m \coloneqq \{A\in \GL_m\mid A^t q_m A=q_m\};
$$
it has two connected components,  with $\SO_m$ being the connected component of the identity.

\begin{definition} \label{definitioncliffordalgebra }
The Clifford algebra 
$$
\Clif_m \coloneqq  \Clif(m, q_m)
$$
is the quotient of the tensor algebra $\bigoplus M^{\otimes \bullet}$
by elements of the form $v\otimes v -q_m(v)$.
\end{definition}

\begin{facts} \cite[Chap. 2]{Mei13}
\begin{itemize}
\item The underlying vector space of $\Clif_m$ is the exterior algebra $\wedge^\bullet M$.
\item There is an involution
$$
\Pi \colon \Clif_m \rightarrow \Clif_m
$$
called the \emph{parity involution}, and we denote by $\Clif_m^+$ the $+1$-eigenspace of even elements.
\end{itemize}
\end{facts}

\begin{definition} \label{definitioncliffordgroup}
Then Clifford group is defined as
$$
\operatorname{C}_m\coloneqq \{g\in \textup{Clif}_m^\times \mid \Pi(g) M g^{-1}\subset M\} 
$$
and 
$$
\SC_m \coloneqq \operatorname{C}_m\cap \Clif_m^+.
$$ 
\end{definition}

By \cite[Sec. 3.1.1]{Mei13} There is a surjective algebra homomorphism, called the \emph{spinor norm}:
\begin{equation}
\label{normmorphism}
\Nm \colon \operatorname{C}_m \longrightarrow  \CC^*.
\end{equation}

\begin{definition} \label{definitionspin}
The Pin group $\Pin_m$ is the kernel of the spinor norm morphism \eqref{normmorphism}.
The Spin group is defined as
$$
\Spin_m \coloneqq  \Pin_m \cap \SC_m,
$$
so we have an exact sequence
\begin{equation}
\label{ex_seq_norm_spin}
    1\longrightarrow \Spin_m \longrightarrow \SC_m \stackrel{\Nm}{\longrightarrow} \CC^* \longrightarrow 1.
\end{equation}
\end{definition}

\begin{facts} \cite[Sec. 3.1.2]{Mei13}
 There is also a group homomorphism $\SC_m \rightarrow \SO_m$ whose kernel is $\CC^*$ and is contained in the center of $\SC_m$. Then one has the following commutative diagram of short exact sequences:
 \begin{equation}
 	\label{exact_seq_spin_SO}
 	\begin{tikzcd}[row sep=large,column sep=large]
 		1 \arrow[r]
 		    & \CC^* \arrow[r]
 		        & \SC_m \arrow[r]
 		            & \SO_m \arrow[r]
 		                & 1 \\
 		1 \arrow[r]
 		    & \Z_2 \arrow[u] \arrow[r]
 		        & \Spin_m \arrow[u] \arrow[r]
 		            & \SO_m \arrow[r] \arrow[u, equal]
 		               & 1
 	\end{tikzcd}
 \end{equation}
\end{facts}

\subsection{Stability}

Let $C$ be a smooth projective curve of genus $g\geq 2$.
The notion of $G$-bundles over $C$ can be translated into the usual vector bundle with additional structure. In our situation, we recall:

\begin{itemize}
	\item  An \emph{orthogonal bundle} is either a principal $\operatorname{O}_m$-bundle, or equivalently a pair $(F,q)$, where $F$ is a vector bundle of rank $m$ and $q\colon S^2 F\rightarrow \sO_C$ is a non-degenerate quadratic form. 
	
	\item A \emph{special (or oriented) orthogonal bundle} is either a principal $\SO_m$-bundle or equivalent a triple $(F,q,\omega)$, where $(F,q)$ is an orthogonal vector bundle and $\omega$ is a section of $\det(E)$ such that $\widetilde{q}(\omega)=1$ with respect to the induced quadratic form $\widetilde{q}\colon S^2 \det(E)\rightarrow \sO_C$.
\end{itemize}

Within these notations, the (semi-)stability translates in a very convenient way \cite[Definition 4.1 and Remark 4.3]{Ram1981}:

\begin{definition}
	\label{defn.semi-stability-orthogonal-bundles}
	Let $\cF$ be an orthogonal bundle $(F,q)$ or a special orthogonal bundle $(F,q,\omega)$. Then $\cF$ is semistable (resp. stable) if every proper non-zero isotropic subbundle of $F$ has degree $\leq 0$ (resp. $<0$).
\end{definition}

By\cite[Proposition 4.2]{Ram1981}, $\cF$ is semistable if and only if $F$ is semistable as a vector bundle. We refer the reader to \cite[Proposition 4.5]{Ram1981} and \cite[4.2]{Serman2012} for a detailed discussion on the stability and regular stability.

\subsection{Clifford, Spin and orthogonal moduli spaces}
\label{ss.CSO-Moduli}

Following \cite{Oxbury1998} and \cite{Serman2012} we gather
some basic facts of moduli spaces of Clifford, Spin and (special) orthogonal bundles.

Note that $\SC_m$ is a connected reductive group \cite[Corollary 2.2]{Oxbury1998}. Restricting the spinor norm morphism \eqref{normmorphism} to $\SC_m$, we obtain an induced  map
\begin{equation}
	\label{norm-moduli}
	\cM_{\SC_m}  \longrightarrow H^1(C,\sO_C^*) \simeq \mbox{Pic}(C), \qquad  [\cF] \ \mapsto \ [\Nm(\cF)]
\end{equation}
which for simplicity of notation we will also denote by $\Nm$. By \cite[Proposition 2.4]{Oxbury1998} and \cite[Thm.5.9]{Ramanathan1996a}, the connected components of $\cM_{\SC_m}$ are parametrised by $\deg(\Nm(\cF))$.  Note that \eqref{exact_seq_spin_SO} defines an action $$
\mbox{Pic}(C) \times \cM_{\SC_m} \longrightarrow \cM_{\SC_m}, \qquad (L,\cF) \mapsto L . \cF
$$
By \cite[(2.3)]{Ram1981} we have for every $L \in \mbox{Pic}(C)$ and any $\SC_m$-bundle $\cF$ that
\begin{equation}
	\label{formula-ramanan}
	\Nm(L . \cF) \simeq L^{2} \otimes \Nm(\cF)
\end{equation}
Using that $J(C)$ is divisible, we see that the fibres of \eqref{norm-moduli} belong to two isomorphism classes distinguished by the degree of $\Nm(\cF)$ \cite[Proposition 2.5]{Oxbury1998} and they are irreducible normal projective varieties. We define
\begin{equation}
    \label{notation-Mspin}
    \cM_{\Spin_m}^+\coloneqq \cM_{\SC_m, [\sO_C]} \simeq \cM_{\Spin_m} \quad \text{and}\quad \cM_{\Spin_m}^- \coloneqq \cM_{\Spin_m,[L]},
\end{equation}
where $\deg(L)$ is odd. Similar to $\cU_C(n,\sO_C(p))$, we can view $\cM_{\Spin_m}^-$ as the moduli space of \emph{twisted $\Spin_m$-bundles}. By \eqref{formula-ramanan}, the spinor norm is invariant by the action of $2$-torsion points
of $\mbox{Pic}(C)$, so $J_2(C)=H^1(C,\Z_2)$
acts on $\cM^{\pm}_{\Spin_m}$. 

Since $\SO_{m}$ is a connected semisimple algebraic group with $\pi_1(\SO_m)\cong \Z_2$, the moduli space $\cM_{\SO_m}$ has two irreducible components distinguished by the \emph{second Stiefel--Whitney class}
$$
w_2 \colon  H^1(C,\SO_m) \longrightarrow H^2(C,\Z_2).
$$
We will denote by $\cM_{\SO_m}^\pm$ the component of $\cM_{\SO_m}$ consisting of $\SO_m$-bundles $\cF'$ with 
\[
w_2(\cF')=(1\mp 1)/2.
\] 
It follows from \cite[(2.10)]{Oxbury1998} that we have Galois covers induced by the natural $J_2(C)$-action:
$$
\mu^{\pm}\colon \cM_{\Spin_m}^{\pm} \stackrel{J_2(C)}{\longrightarrow} \cM_{\SO_m}^{\pm}.
$$

There also exists a moduli space $\cM_{\operatorname{O}_m}$ for semistable $\operatorname{O}_m$-bundles \cite[1.10]{Serman2012}, which has several connected components determined by the 
first and second Stiefel--Whitney classes $w_i(\cF'')\in H^2(X,\ZZ_2)$, $i=1,2$, where $\cF''=(F,q)$ is an $\operatorname{O}_m$-bundle. 
Note that $w_1(\cF'')$ is nothing but the determinant $\det(F)$ \cite[2.4]{Serman2012}. Moreover, we have the natural forgetful map
\begin{equation}
	\label{forgetSO}
	\cM_{\SO_m} \longrightarrow \cM_{\operatorname{O}_m}, \qquad (F, q, \omega) \ \mapsto \ (F, q)
\end{equation}
Denote by $\cM_{\operatorname{O}_m}^{\pm}$ the images of $\cM_{\SO_m}^{\pm}$. They are irreducible components of $\cM_{\operatorname{O}_m}$. By \cite[2.8 and Proposition 2.9]{Serman2012}, The maps $\cM_{\SO_m}^{\pm}\rightarrow \cM_{\operatorname{O}_m}^{\pm}$ are of degree two and the compositions
\begin{equation}
	\label{define-eta}
	\nu^{\pm}\colon \cM_{\Spin_{m}}^{\pm}\longrightarrow \cM_{\SO_{m}}^{\pm} \longrightarrow \cM_{\operatorname{O}_{m}}^{\pm}
\end{equation}
are $(\ZZ_2)^{m+1}$-Galois covers of their images. 

\begin{example}[\protect{\cite[4.4]{Oxbury1999}}]
\label{ex.Spin4}
    It is well-known that the special Clifford group $\SC_4$ is isomorphic to the subgroup of $\GL_2\times \GL_2$ consisting of pairs of matrices $(A,B)$ such that $\det(A)=\det(B)$ so that the spinor norm is then the common $2\times 2$ determinant. Then $\Spin_4=\SL_2\times \SL_2$.  Thus $\cM_{\SC_4}=\cM_{\GL_2}\times_{\det} \cM_{\GL_2}$ while $\cM^{\pm}_{\Spin_4}$ are 
    \[
    \cU_C(2,\sO_C)\times \cU_C(C,\sO_C) = \cM_{\SL_2} \times \cM_{\SL_2}\quad \text{and}\quad \cU_C(2,L)\times \cU_C(2,L)
    \]
    respectively, where $L$ is a line bundle of odd degree over $C$.
\end{example}

\section{The intersection of two quadrics}
\label{section:twoquadrics}

We will now restrict our attention to the odd-dimensional intersection of two quadrics:

\begin{setup}{\rm
    \label{setup-odd}
  Let $X \subset \PP^{2g+1}$ be the complete intersections of two quadrics given by the following two equations:
 \[
 q_1\coloneqq \sum_{j=0}^{2g+1} x_j^2 \quad \text{and}\quad q_2\coloneqq \sum_{j=0}^{2g+1} \lambda_j x_j^2,
 \]
with pairwise distinct coefficients $\lambda_j \in \CC$.
Assume that $g \geq 2$ and set $n=2g-1$, so $\dim X=n$.
We identify the pencil $\langle q_1,q_2\rangle \subset |\sO_{\PP^{2g+1}}(2)|$ with $\PP^1$ such that $t=[a:b]\in \PP^1$ corresponds to $q_t=a q_1 + b q_2$.

Let 
$$
\pi\colon C\rightarrow \PP^1
$$ 
be the associated hyperelliptic curve of genus $g$, i.e., 
$\pi$ is the double cover of $\PP^1$ branching exactly over
the points $r_j=[\lambda_j:-1]\in \PP^1$ for $j=0, \ldots, 2g+1$.
We denote the branch locus by $\Delta \subset \PP^1$ and 
the Weierstra\ss\ points $W\subset C$ by
$$
W=\{p_0,\dots,p_{2g+1}\}
$$
where $r_j=\pi(p_j)$. As usual denote by 
$i\colon C\rightarrow C$ the hyperelliptic involution.

Let $h$ be the pull-back $\pi^*\sO_{\PP^1}(1)$ and set
\begin{equation}
    \label{notation-alpha}
    \alpha \coloneqq h^{g-1}\otimes \sO_{C}(p_{2g+1}).
\end{equation}
Then $\alpha$ has odd degree and applying the construction of twisted Spin bundles as in \eqref{notation-Mspin}  we write
$$
\cM_{\Spin_{2g}}^- \coloneqq \cM^-_{\Spin_{2g},[\alpha]}
$$
}
\end{setup}

\begin{notation}
	\label{notation.inv-antiinv}
	For a vector space $M$ equipped with an involution, we shall denote by $M^{-}$ (resp. $M^+$) the eigenspace with respect to $-1$ (resp. $+1$).
\end{notation}

For the rest of the section we will always assume that we are in
Setup \ref{setup-odd} as above.

\begin{definition} \label{defn:type}
    Let $(F,q)$ be an orthogonal bundle of rank $m$ over $C$.
    \begin{enumerate}
    	\item We call $(F,q)$ $i$-invariant if there exists an isomorphism $(i^*F,i^*q)\cong (F,q)$.
    	
    	\item Assume that $(F,q)$ is $i$-invariant. The type $\tau\coloneqq (\tau_0,\dots,\tau_{2g+1})$ of $(F,q)$ is defined with $\tau_j=\dim F_{p_j}^-$ with respect to the lifted action $i$ on $F_{p_j}$.
    \end{enumerate}
\end{definition}

If $(F,q)$ is an $i$-invariant orthogonal bundle, then changing the sign of the isomorphism $(i^*F,i^*q)\cong (F,q)$  transforms the type  $\tau$ into $\tau'=(\tau'_0,\dots,\tau'_{2g+1})$, where $\tau'_j=m-\tau_j$ for all $j$. We denote by 
$$
\cM_{\operatorname{O}_m}^{\pm,\tau} \subset \cM_{\operatorname{O}_m}^{\pm}
$$
the subset corresponding to the orthogonal bundles of type $\tau$ or $\tau'$. We define
\[
\cM_{\Spin_m}^{\pm,\tau} \coloneqq  (\nu^{\pm})^{-1}(\cM_{\operatorname{O}_m}^{\pm,\tau}) \quad \text{and} \quad \cM_{\SO_m}^{\pm,\tau}\coloneqq \mu^{\pm}(\cM_{\Spin_m}^{\pm,\tau}).
\].

\begin{notation} \label{notation:MandU}
	In this paper, we will focus on a specified type $\tau=(1^{2g+1},2g-1)$, so we introduce the following spaces in this special case:
	\[
	\cU \coloneqq \mathcal{\cM}_{\operatorname{O}_{2g}}^{-,\tau},\quad \cV \coloneqq \cM_{\SO_{2g}}^{-,\tau}\quad \text{and} \quad \cM \coloneqq \cM_{\Spin_{2g}}^{-,\tau}.
	\]
	For simplicity, we denote by 
    $$
    \mu \colon \cM \longrightarrow \cV \qquad \text{and}\qquad
    \nu \colon \cM \longrightarrow \cU
    $$ 
    the finite morphisms induced by $\mu^-$ and $\nu^-$, respectively.
\end{notation}

We report in our notation the main result we need from \cite{Ram1981}:

\begin{theorem}[\protect{\cite[Theorem 3]{Ram1981}}]
\label{thm.RamananThm}
    There exists a (non-canonical) isomorphism
    \begin{equation}
    	\label{eq.Ramanan-Isom}
    	\cM \cong X.
    \end{equation}
\end{theorem}

The isomorphism depends on the choice of an $i$-invariant square root of $h^{2g-1}$. Throughout this paper, we will work with the isomorphism determined by the line bundle $\alpha$ from \eqref{notation-alpha}.
We will not reprove Ramanan's theorem, but for the purpose of our proof we will present below the construction of some of the related morphisms. 

\subsection{Partitions of $W$ and quotients of $\cM$}
\label{ss.Quotient-Isom}

We will now describe the quotient maps $\mu$ and $\nu$. We start by recalling the following result.

\begin{lemma}[\protect{\cite[Lemma 2.1]{DesaleRamanan1976}}]
\label{lem.PartitionW}
    There is a one-to-one correspondence between line bundles $\eta$ of fixed even (resp. odd) degree such that $\eta^2\cong h^{\deg\eta}$, and the set of partitions $W=S\cup T$ of the set of Weierstra\ss\ points into subsets of even (resp. odd) cardinality.
\end{lemma}

In fact the action of $i$ can be lifted to $\eta$. Then $S$ (resp. $T$) is the subset of $W$ consisting of points $p\in W$ such that $i_p\colon \eta_p\rightarrow \eta_p$ acts as $+1$ (resp. $-1$).

The group $(\ZZ_2)^{2g+2}$ can be naturally identified to the set consisting of all ordered partition $(S,T)$ of $W$ by sending $a=(a_j)_{0\leq j\leq 2g+1}\in (\Z_2)^{2g+2}$ to the partition:
\[
S\coloneqq \{p_j \in W\mid a_j=0\}\quad \text{and} \quad T\coloneqq \{p_j \in W \mid a_j=1\}.
\]

Let $\Upsilon\cong (\Z_2)^{2g+1}$ be the quotient $(\Z_2)^{2g+2}$ by $-1$ and let $\Upsilon^+\subset \Upsilon$ be the subgroup of elements that correspond to even partitions of $W$. Then Lemma \ref{lem.PartitionW} yields a canonical isomorphism 
\[
\Upsilon^+ \cong J_2(C).
\]

By Lemma \ref{lem.PartitionW}, we can canonically identify $\Upsilon$ to $J_2(C)\cup P$ as sets, where $P$ consists of $i$-invariant line bundles $\xi$ such that $\xi^2\cong h^{2g-1}$. For each $\xi\in P$, we define an action on $\cM_{\Spin_{2g}}^{-,\tau}$ as follows:
\begin{equation}
	\label{eq.Gamma-action-cM}
	\xi\colon \cM \longrightarrow \cM,\quad [\cF] \mapsto [(\xi^*\otimes \alpha) . i^*\cF].
\end{equation}
 This is well-defined by \eqref{formula-ramanan}, 
 moreover the element $\alpha \in P$ acts by  $[\cF] \mapsto [i^*\cF]$.
 Then $\mu$ is exactly the quotient map under the $\Upsilon^+=J_2(C)$ action so that
\[
\cV \cong \cM/\Upsilon^+.
\]
Then $\Upsilon/\Upsilon^+\cong \Z_2=\langle i\rangle$ acts on $\cV$ by sending $[\cF']\in \cV$ to $[i^*\cF']$ so that
\[
\cV/\langle i\rangle \cong \cU.
\]

On the other hand, the group $\Upsilon$ acts on $X$ by changing signs of the coordinates such that $\Upsilon^+$ contains elements that change an even number of coordinates. Thus the isomorphism \eqref{eq.Ramanan-Isom} is $\Upsilon$-equivariant,
and we have:

\begin{theorem}[\protect{\cite[Theorem 1]{Ram1981}}]
\label{theorem:ramanan5}
  We have isomorphisms
    \[
    \cV \cong X/\Upsilon^+ 
    \quad \text{and} \quad  
    \cU \cong X/\Upsilon.
    \]
\end{theorem}

The geometry of the quotients is well-understood:

\begin{corollary}
We have $\cU \cong \PP^{2g-1}$ and $\cV\cong Z$ where $Z$ is the subvariety of the weighted projective space $\PP(1^{2g+2},g+1)$ defined by the three equations
	\[
	\sum_{j=0}^{2g+1} y_j =0, \quad \sum_{j=0}^{2g+1} \lambda_j y_{j} = 0, \quad y_{2g+2}^2 - \prod_{j=0}^{2g+1} y_j = 0,
	\]
	where $[y_0:\dots:y_{2g+2}]$ are the weighted homogeneous coordinates on $\PP(1^{2g+1},g+1)$.
\end{corollary}

\begin{proof}
	This follows from the fact that the following map
	\[
	X\subset \PP^{2g+1} \longrightarrow \PP^{2g+1}, \quad [x_0:\dots:x_{2g+1}] \longmapsto [x_0^2:\dots:x_{2g+1}^2]
	\]
	is the quotient of $X$ by the $\Upsilon$-action, and it factors through the map
	\[
	X\subset \PP^{2g+1} \longrightarrow \PP(1^{2g+2},g+1),\quad [x_0:\dots:x_{2g+1}] \longmapsto \left[x_0^2:\dots:x_{2g+1}^2:\prod_{j=0}^{2g+1} x_j \right],
	\]
    which is the quotient by $\Upsilon^+$.
\end{proof}

\subsection{Description of the morphism $X \rightarrow \cU$}
\label{ss.X-to-OrthBundles}

Following \cite{Ram1981}, we explain how to construct an  $i$-invariant orthogonal bundle $(F,q)$ of rank $2g$ and type $\tau$ starting with a point in $X$. 

\begin{notation} \label{notationU}
For use in Subsection \ref{ss.cM-to-X}, let us refine the notation in the Setup \ref{setup-odd}: for each $0\leq j\leq 2g+1$, let $U_j$ be the $j$-th coordinate line with $q^j\colon S^2 U_j\rightarrow h_{p_j}^{2g-1}\cong \CC$ (later on it will become clear why we have specified $h_{p_j}^{2g-1}$ instead of simply $\CC$) a fixed non-degenerate quadratic form.
We write
$$
U \coloneqq  \bigoplus_{j=0}^{2g+1} U_j \cong \CC^{2g+2}.
$$
\end{notation}

Then $X\subset \PP^{2g+1}=\PP U^*$ and the pencil $\langle q_1,q_2\rangle = \PP^1$ can be seen as a morphism
$$
\holom{\bar q}{S^2 (U \otimes \sO_{\PP^1})}{\sO_{\PP^1}(1)},
$$
so that in every point $t=[a:b] \in \PP^1$ the corresponding quadratic form is given by $q_t = a q_1 + b q_2$.
The morphism $\bar q$ 
induces a morphism 
$$
U \otimes \sO_{\PP^1} \longrightarrow U^* \otimes \sO_{\PP^1}(1),
$$
which is an isomorphism outside $\Delta$. For any $0\leq j\leq 2g+1$, let $r_j = [\lambda_j:-1]$ be the $j$-th branch point.
The quadratic form $q_{r_j}=\lambda_j q_1-q_2$ is degenerate with
kernel the $j$-th coordinate line so we obtain an exact sequence of vector spaces
\[
0\longrightarrow \ell_j  \longrightarrow U \longrightarrow  U^*,
\]
where $\ell_j$ is the $j$-th coordinate line, i.e. the vertex of the singular hyperquadric $\{ \lambda_j q_1-q_2 = 0 \} \subset \PP U^*$.

Let now $[V] \in X$ be a point, so that $V \subset U$ is a $1$-dimensional subspace isotropic with respect to any quadratic form in the pencil. Then 
$V \otimes \sO_{\PP^1} \subset U \otimes \sO_{\PP^1}$ is a trivial 
subbundle that is contained in $\ker \bar q$. We denote by
$$
V^\perp \subset U \otimes \sO_{\PP^1}
$$
the orthogonal to $V \otimes \sO_{\PP^1}$ with respect to $\bar q$. 

Since $X$ is smooth, the subspace $\ell_j$ is not contained in $V$. Now a straightforward computation shows that $V^{\perp}$ has constant rank $2g+1$ over $\PP^1$. 
Moreover since the quadratic form $\bar q$ has values in $\sO_{\PP^1}(1)$
we obtain that $V^{\perp}$ has degree $-1$.
Therefore
$$
V^\perp /(V \otimes \sO_{\PP^1}) \subset U/V \otimes \sO_{\PP^1}
$$
is a vector bundle on $\PP^1$ of rank $2g$ and degree $-1$.

\begin{lemma}
\label{lem_trivial_factor}
For any $[V] = x \in X$ we have:
$$
V^\perp/(V\otimes \sO_{\PP^1})\cong \sO_{\PP^1}^{\oplus 2k-1}\oplus \sO_{\PP^1}(-1).
$$
In fact the trivial part corresponds to $T_{X,x}$, i.e. we have a natural isomorphism
$$
V^\perp/(V\otimes \sO_{\PP^1}) 
\simeq (T_{X,x} \otimes V) \otimes \sO_{\PP^1} \oplus \sO_{\PP^1}(-1).
$$
\end{lemma}

\begin{remark} 	\label{rem.Trivial-factor}
Let us recall that given a smooth quadric 
$$
Q' = \{ q' = 0 \} \subset \PP U^*
$$
and a point $[V] = x \in Q'$, the tangent space
$T_{Q',x}$ is canonically isomorphic to $(V^{\perp_{q'}}/V) \otimes V^*$. Thus if $[V] = x \in X$ we have a canonical isomorphism
$$
T_{X,x} \simeq  (V^{\perp_{q_1}} \cap V^{\perp_{q_2}}/V) \otimes V^*.
$$
\end{remark}

\begin{proof}[Proof of Lemma \ref{lem_trivial_factor}]
Since $q_t = a q_1 + b q_2$ and $V \subset U$ is isotropic for
both $q_1$ and $q_2$ we have a natural inclusion
$$
(V^{\perp_{q_1}} \cap V^{\perp_{q_2}}) \otimes \sO_{\PP^1}
\subset 
V^\perp \subset U \otimes \sO_{\PP^1}.
$$
Moreover $V \subset V^{\perp_{q_1}} \cap V^{\perp_{q_2}}$, so
$$
(V^{\perp_{q_1}} \cap V^{\perp_{q_2}})/V \otimes \sO_{\PP^1}
\subset V^\perp /(V \otimes \sO_{\PP^1})
$$
is a trivial subbundle of rank $2g-1$.
Since $V^\perp /(V \otimes \sO_{\PP^1})$ has degree $-1$ we have a splitting extension 
$$
0 \rightarrow (V^{\perp_{q_1}} \cap V^{\perp_{q_2}})/V \otimes \sO_{\PP^1} \rightarrow V^\perp /(V \otimes \sO_{\PP^1})
\rightarrow
\sO_{\PP^1} (-1) \rightarrow 0.
$$
The second statement follows from Remark \ref{rem.Trivial-factor}.
\end{proof}

\begin{notation} \label{notationNT}
	For simplicity of notation, we set 
    $$
    \widetilde{N} \coloneqq  V^{\perp}/(V\otimes \sO_{\PP^1})
    $$
and denote by $\tilde T \subset \tilde N$ its trivial factor.
We also denote by
    $$
  T\coloneqq  \pi^*  \tilde T \subset \pi^* \tilde N \eqqcolon  N
    $$
    the induced inclusion on the pull-back.
\end{notation}

Let $\holom{\widetilde{q}}{S^2 \widetilde{N}}{\sO_{\PP^1}(1)}$ be the quadratic form induced by $\bar q$; then $\widetilde{q}$ degenerates
exactly in the branch points $\Delta \subset \PP^1$ and we have an exact sequence
$$
0 \longrightarrow \widetilde{N} \longrightarrow
\widetilde{N}^* \otimes \sO_{\PP^1}(1)
\longrightarrow \bigoplus_{r_j \in \Delta} \CC_{r_j} \longrightarrow 0 
$$
Pulling back to $C$ we obtain
$$
0 \longrightarrow N \longrightarrow N\otimes h \longrightarrow \bigoplus_{r_j \in \Delta} \pi^* \CC_{r_j}  \longrightarrow 0 
$$
Twisting with $\sO_C(-p_{2g+1})$ we obtain an exact sequence
\begin{equation} \label{inclusionNNstar}
\begin{split}
0 \longrightarrow N \otimes \sO_C(-p_{2g+1}) \longrightarrow N^* \otimes \sO_C(p_{2g+1})  \longrightarrow  \bigoplus_{r_j \in \Delta} \pi^* \CC_{r_j}  \longrightarrow 0 .
\end{split}
\end{equation}
Notice that $\pi^* \CC_{r_j}  \simeq \sO_C/\sO_C(-2p_j)$, so it has a natural
map to $\CC_{r_j}$ (seen as the quotient $\sO_C/\sO_C(-p_j)$), and this map identifies to the projection
$$
\bigoplus_{r_j \in \Delta} \pi^* \CC_{r_j}  \longrightarrow
\left(\bigoplus_{r_j \in \Delta} \pi^* \CC_{r_j}  \right)^i
$$
onto the $i$-invariant part. Consider now the surjective morphism
$$
\beta\colon  N^* \otimes \sO_C(p_{2g+1}) 
\longrightarrow \bigoplus_{r_j \in \Delta} \pi^* \CC_{r_j} 
\longrightarrow
\left(\bigoplus_{r_j \in \Delta} \pi^* \CC_{r_j} \right)^i
\longrightarrow 0.
$$

\begin{prop} \label{prop.defineF} 
Let $F \rightarrow C$ be the kernel of the morphism $\beta$.
Then the quadratic form $\widetilde q$ induces on $F$ the structure of an $i$-invariant orthogonal bundle of rang $2g$ and type $\tau=(1^{2g+1},2g-1)$.
Moreover  $\cF\coloneqq (F,q)$ is semistable as an $\operatorname{O}_{2g}$-bundle. 

In conclusion we have constructed a morphism
$$
X \longrightarrow \mathcal U, \qquad x=[V] \ \longmapsto [(F,q)].
$$
\end{prop}

\begin{proof}
For the convenience of the reader we justify in detail
that $(F, q)$ is an orthogonal bundle, the semistability
is shown in \cite[Proposition 5.6]{Ram1981}.

The quadratic form $\widetilde q$ induces an inclusion $
    N \otimes \sO_C(-p_{2g+1}) \hookrightarrow N^* \otimes \sO_C(p_{2g+1})$, which we denote again by $\widetilde{q}$. By construction the morphism factors into a sequence of (strict) inclusions 
    $$
    N \otimes \sO_C(-p_{2g+1}) \hookrightarrow F \stackrel{\kappa}{\hookrightarrow}  N^* \otimes \sO_C(p_{2g+1})
    $$
    The dual of the inclusion $N \otimes \sO_C(-p_{2g+1}) \subset F$ is an inclusion $F^*\hookrightarrow N^* \otimes \sO_C(p_{2g+1})$, whose cokernel is the same as the cokernel of $\kappa:F \hookrightarrow N^* \otimes \sO_C(p_{2g+1})$. Thus the two maps determine 
    an isomorphism $$q:F^*\to F.$$
    Finally, the bundle $N^* \otimes \sO_C(p_{2g+1})$ is $i$-invariant by construction and $F \subset N^* \otimes \sO_C(p_{2g+1})$ is the kernel of
the $i$-equivariant map $\beta$, so $F$ is $i$-invariant.
\end{proof}

\begin{lemma}
	\label{lem.stability-F}
	If the point $x=[V] \in X$ is contained in one of the coordinate hyperplanes $x_i=0$, the vector bundle $F$ defined by Proposition \ref{prop.defineF} is not stable.
\end{lemma}

\begin{proof}
	Up to renumbering, we may assume that $V$ is contained in the hyperplane $\{x_0=0\}$. 
    The coordinate line $\ell_0$ is distinct from the subspace $V \subset U$, so its image
    $\bar{\ell}_0\subset U/V$ is a one-dimensional subspace and 
\[
	\widetilde{L}\coloneqq \bar{\ell}_0\otimes \sO_{\PP^1} \subset U/V \otimes \sO_{\PP^1}
	\]
    is a rank one subbundle.
    Since $[V] \in \{ x_0=0 \}$ the subspace $\ell_0$ is contained in $V^{\perp_{q_{t}}}$ for any $t\in \PP^1$. Thus we have $\widetilde L \subset \widetilde T \subset U/V \otimes \sO_{\PP^1}$.
	
 Set $L\coloneqq \pi^*\widetilde{L}$. Then we have injections of sheaves
	\[
	h\colon L \longrightarrow T \longrightarrow N \longrightarrow F^*\otimes \sO_C(p_{2g+1})
	\]
	and $h$ vanishes exactly over $p_0$. Let $\bar{L}$ be the saturation of $L$ in $F^*\otimes \sO_{C}(p_{2g+1})$. Then there exists a non-zero morphism $L\otimes \sO_{C}(p_0)\rightarrow  \bar{L}$, which implies
	\[
	\deg(\bar{L}) \geq \deg(L) + 1 = 1.
	\]
	Therefore $F^*\otimes \sO_C(p_{2g+1})$ is not stable and so is $F$.
\end{proof}

\begin{lemma}
	\label{lem.dim-h0-tangent}
Given a  point $x=[V] \in X$, let $(F,q)$ be the orthogonal bundle given by Proposition \ref{prop.defineF}. Then one has
	$$
    \dim H^0(C,\wedge^2 F\otimes K_C)^+ \geq  2g-1.
    $$
\end{lemma}

\begin{proof}
	By the Riemann--Roch formula for vector bundles over curves and \cite[Proposition 2.2]{DesaleRamanan1976}, we have
	\[
	\dim H^0(C,\wedge^2 F \otimes K_C)^+ - \dim H^1(C,\wedge^2 F\otimes K_C)^+ = \frac{1}{2}\left(d + \sum_{p\in W} r_p\right) - rg,
	\]
	where $r$ is the rank of $\wedge^2 F$, $d$ is the degree of $\wedge^2 F\otimes K_C$ and $r_p$ is the dimension of $(\wedge^2 F\otimes K_C)_p^+$. As $F$ is of type $(1^{2g+1},2g-1)$ and $K_C$ is of type $((-1)^{2g+2})$, we get $r_p=2g-1$ for any $p\in W$. Then the required result follows immediately as $r=g(2g-1)$ and $d=g(g-1)(2g-2)$.
\end{proof}

\begin{notation}
\label{notation.E}
Given a  point $x=[V] \in X$, let $(F,q)$ be the orthogonal bundle given by Proposition \ref{prop.defineF}.  Then we set
$$
E\coloneqq  F \otimes \alpha
$$
where $\alpha$ is the $i$-invariant line bundle from Setup \ref{setup-odd}. 
By \eqref{inclusionNNstar} we have an inclusion $N \otimes \sO_{C}(-p_{2g+1}) \subset F$,
so $E$ is given by the extension 
\begin{equation}
\label{eq_ext_E}
    0\longrightarrow N\otimes h^{g-1} \longrightarrow E \longrightarrow \bigoplus_{{p_j}\in W} E_{p_j}^- \longrightarrow 0.
\end{equation}
In particular we have a map
\begin{equation}
    \label{defineiota}
    \iota\colon N\longrightarrow E\otimes h^{-(g-1)},
\end{equation}
which is an isomorphism outside $W$ and has rank $2g-1$ over $p_j\in W$ for all $j$.
\end{notation}

By Lemma \ref{lem_trivial_factor} the vector bundle $N$ contains a trivial subbundle $T$ induced by $T_{X,x}\otimes V$. For later use we observe:

\begin{lemma}
	\label{lem.inj-iota}
	Assume that $x = [V] \in X\subset \PP U^*$ is not contained in any coordinate hyperplane $\{ x_i=0\}$. Then for any point $y\in C$, the restriction
	\[
	\iota_y|_{T_y}\colon T_y \longrightarrow (E\otimes h^{-(g-1)})_y
	\]
	is injective.
\end{lemma}

\begin{proof}
	Since $\iota_y$ is an isomorphism if $y\not\in W$, it remains to consider the case where $y\in W$. Without loss of generality, we may assume $y=p_0$. The kernel of the linear map
	\[
	\iota_{p_0}\colon N_{p_0} \longrightarrow (E\otimes h^{-(g-1)})_{p_0}
	\]
    is the image of the coordinate line $\bar l_0 \subset U/V$ (see the proof of Lemma \ref{lem.stability-F} for the construction). 
	Since $x$ is not contained in $\{x_0=0\}$ we have $\bar l_0 \not\subset T_{X,x}\otimes V$, so the restriction $\iota_{r_0}|_{T_{r_0}}$ is injective.
\end{proof}

\subsection{Description of the morphism $\cM \rightarrow X$}
\label{ss.cM-to-X}

Let us now go backwards and construct a point of $[V]\in X$ 
from a point $[\cF]\in \cM$. Let $(F,q) \coloneqq \nu([\cF]) \in \mathcal U$ be the associated orthogonal bundle, and consider the bundle $E\coloneqq F\otimes \alpha$
(cf. Notation \ref{notation.E}).

The quadratic form $q$ induces a quadratic form $S^2 E\rightarrow \alpha^2 \cong h^{2g-1}$, which for simplicity we denote again by $q$.
Then $E_{p_j}^{-}\cong \CC$ for any $0\leq j\leq 2g+1$ and $q|_{E_{p_j}^-}$ is non-degenerate.

As shown in the proof of \cite[Theorem 2]{Ram1981} the spin structure on $\cF$
determines a family $\{\epsilon_j\}_{0\leq j\leq 2g+1}$, which is unique up to $-1$, of orthogonal isomorphisms
\[
\epsilon_j\colon (E_{p_j}^-,q|_{E_{p_j}^-}) \longrightarrow (U_j,q^j),
\]
where $(U_j,q^j)$ are the fixed quadratic forms introduced in Notation \ref{notationU}.
Composing the quadratic map
$$ 
\bigoplus_{j=0}^{2g+1} q^j\colon  S^2 U \longrightarrow \bigoplus_{j=0}^{2g+1} h^{2g-1}_{p_j}\cong H^0(\PP^1,\sO_{\PP^1}(2g-1)\otimes \sO_{\Delta}) 
$$
with the coboundary map in cohomology of the exact sequence
\[
0\longrightarrow \sO_{\PP^1}(-3) \longrightarrow \sO_{\PP^1}(2g-1) \longrightarrow \sO_{\PP^1}(2g-1)\otimes \sO_{\Delta}\longrightarrow 0
\]
yields a map
$$ 
S^2 U \rightarrow H^0(\PP^1,\sO_{\PP^1}(2g-1)\otimes \sO_{\Delta})\twoheadrightarrow H^1(\PP^1,\sO_{\PP^1}(-3)) \cong \CC^2, 
$$
which defines a pencil of quadrics $\langle q_1,q_2\rangle$ on $U$. 

\begin{lemma}
	The complete intersection $X\subset \PP U^*$ defined by $q_1=q_2=0$ is isomorphic to the one defined at the beginning of this section. 
\end{lemma}

\begin{proof}
	We need to prove that there exist trivialisations $\sO_{\PP^1}(2g-1)_{r_j}\cong \CC$ such that the image of $H^0(\PP^1,h^{2g-1})\cong \CC^{2g}$ in $\sO_{\PP^1}(2g-1)\otimes \sO_{\Delta}\cong \CC^{2g+2}$ is the subspace defined by the following two linear forms:
	\[
	f_1\coloneqq \sum_{j=0}^{2g+1} z_j \quad \text{and} \quad f_2 \coloneqq \sum_{j=0}^{2g+1} \lambda_j z_j,
	\]
	where $z_j$ is the coordinate on $\sO_{\PP^1}(2g-1)_{r_j}\cong \CC$. Let 
	\[
	s_k=(-1)^{k+1} x_0^k x_1^{2g-1-k},\quad 0\leq k\leq 2g-1
	\]
	be a basis of $H^0(\PP^1,\sO_{\PP^1}(2g-1))$, where $[x_0:x_1]$ is the homogeneous coordinates on $\PP^1$. Then the coordinate $r_j=[\lambda_j:-1]$ yields a trivilisation $\sO_{\PP^1}(2g-1)_{r_j}\cong \CC$ such that the image of $s_k$ corresponds to the vector
	\[
	v_k\coloneqq (\lambda_0^k,\dots,\lambda_{2g+1}^k) \in \CC^{2g+2}.
	\]
	It remains to show that there exists a diagonal matrix $A\coloneqq \diag(a_0,\cdots,a_{2g+1})$ such that $a_j\not=0$ for any $0\leq j\leq 2g+1$ and $A(v_k)$ satisfies $f_1=f_2=0$ for any $0\leq k\leq 2g-1$. The latter condition is equivalent to say that $(a_0,\dots,a_{2g+1})\in \CC^{2g+2}$ is contained in the kernel of the following linear map
	\[
	B\coloneqq 
    \begin{pmatrix}
	    1   &  \dots & 1 \\
        \lambda_0 & \dots & \lambda_{2g+1} \\
        \vdots    & \vdots & \vdots \\
        \lambda_0^{2g} & \dots & \lambda_{2g+1}^{2g}
	\end{pmatrix}
    \colon \CC^{2g+2} \longrightarrow \CC^{2g+1}.
	\]
	Since the $\lambda_j$'s are pairwise distinct, the Vandermonde determinant  implies that $B$ has rank $2g+1$. In particular, there exists a unique  (up to scaling) non-zero vector $v=(a_0,\dots,a_{2g+1})\in \CC^{2g+2}$ contained in $\ker(B)$. Moreover, if one of the $a_j$'s is zero, saying $a_{2g+1}=0$ for instance, then the non-zero vector $v'=(a_0,\dots,a_{2g})$ is contained in the kernel of the following linear map
    \[
    B'\coloneqq 
    \begin{pmatrix}
	    1   &  \dots & 1 \\
        \lambda_0 & \dots & \lambda_{2g} \\
        \vdots    & \vdots & \vdots \\
        \lambda_0^{2g} & \dots & \lambda_{2g}^{2g}
	\end{pmatrix}
    \colon \CC^{2g+1} \longrightarrow \CC^{2g+1},
    \]
    which is absurd because $B'$ is an isomorphism by the Vandermonde determinant.
\end{proof}

By Proposition \ref{prop.defineF} the vector bundle $E^*\otimes K_C$ is semistable.
Since $\deg(E^*\otimes K_C)=-2g$, we obtain 
\[
H^1(C,E)\cong H^0(C,E^*\otimes K_C)=0.
\]
which implies $H^1(C,E)^-=0$, hence $\dim H^0(C,E)^{-}=1$ by \cite[Proposition 2.2]{DesaleRamanan1976}. Moreover, by \cite[Proposition 4.11]{Ram1981}, the evaluation map
\begin{equation}
	\label{eq.Embedding-H0(C,E)-}
    \begin{tikzcd}[column sep=large]
        H^0(C,E)^{-} \arrow[r, "{(\operatorname{ev}_{p_j})_j}"] 
            & \bigoplus_{j=0}^{2g+1} E_{p_j}^- \arrow[r, "{(\epsilon_j)_j}", "\cong"']
                & \bigoplus_{j=0}^{2g+1} U_j = U
    \end{tikzcd}
\end{equation}
is injective and the image of $H^0(C,E)^{-}$ in $U$ is a line isotropic with respect to $q_1$ and $q_2$. In conclusion we have constructed a morphism
\begin{equation}
    \label{morphismMtoX}
    \cM \longrightarrow X, \qquad [\cF] \ \longmapsto \ [H^0(C,E)^{-}],
\end{equation}
which is the isomorphism of \cite[Thm.3]{Ram1981}. Note that the composition of this isomorphism with the morphism $X \rightarrow \cU$ from Proposition \ref{prop.defineF} identifies to $\nu: \cM \rightarrow \cU$.

\begin{example}
    Let $C$ be a hyperelliptic curve of genus $2$. As explained in Example \ref{ex.Spin4}, the moduli space $\cM_{\Spin_4}^-$ is isomorphic to $\cU_C(2,\alpha)\times \cU_C(2,\alpha)$. Then the morphism $\nu^-\colon \cM_{\Spin_4}^-\rightarrow \cM_{\operatorname{O}_4}^-$ is given by
    \[
    (F_1,F_2) \longmapsto F_1\otimes F_2^*.
    \]
    The subvariety $\cM$ of $\cM^-_{\Spin_4}$ thereby consists of pairs $(F_1,F_2)$ with $F_2\cong i^*F_1$, which allows us to recover the isomorphism $X\cong \cU_C(2,\alpha)$ provided in \cite{Newstead1968}.
\end{example}

\section{The odd dimensional case}
\label{s.Geometric-inter-hM}

This section is devoted to give a characterisation of the cotangent bundle of $\cM$, which allows us to define a Hitchin morphism $h_{\cM}$ for $\cM$. Then we show how $h_{\cM}$ can be related to the morphism $\Phi_X$ via a geometric interpretation. We will work in the Setup \ref{setup-odd}.

\subsection{Skew-symmetric maps}
\label{ss.skew-symmetric-map}

We recall some basic facts concerning skew-symmetric maps. Let $M$ be an $m$-dimensional complex vector space equipped with a non-degenerate symmetric bilinear form $q_M\colon S^2 M\rightarrow \CC$. A linear map $\theta_M\colon M\rightarrow M$ is said to be \emph{skew-symmetric with respect to $q_M$} if we have
\begin{equation}
	\label{eq.Skew-symmetric-defn}
	q_M(\theta_M(v_1),v_2) = - q_M(v_1,\theta_M(v_2))
\end{equation}
for any $v_1$, $v_2\in M$. If we choose the trivialisation $M\cong \CC^m$ such that $q_M$ becomes the standard quadratic form on $\CC^m$, then $\theta_M$ can be written as a skew-symmetric matrix. Moreover, by \eqref{eq.Skew-symmetric-defn} and a dimension counting, the image $\operatorname{Im}(\theta_M)$ is exactly the orthogonal complement of $\ker(\theta_M)$, i.e.,
\[
\ker(\theta_M)^{\perp_{q_M}} = \operatorname{Im}(\theta_M).
\]

\begin{lemma}
	\label{lem.Nilpotent-orth-decomp}
	Let $(M,q_M)$ be an $m$-dimensional complex vector space equipped with a non-degenerate symmetric bilinear form and let $\theta_M\colon M\rightarrow M$ be a skew-symmetric linear map with respect to $q_M$. If $\theta_M$ is not nilpotent and is of rank two, then there exists an orthogonal decomposition 
	\[
	M = \ker(\theta_M) + \operatorname{Im}(\theta_M).
	\]
\end{lemma}

\begin{proof}
	Choosing the trivialisation $M\cong \CC^m$ such that $\theta_M$ is given by a skew-symmetric matrix. The non-zero eigenvalues of $\theta_M$ occur in  pairs $\pm \zeta_j$. In particular, since $\theta_M$ has rank two, $\theta_M$ is non-nilpotent if and only if $\theta_M$ admits a unique pair of non-zero eigenvalues $\pm \zeta\in \CC$, which then implies that $\theta_M$ is diagonalisable. Thus there exist non-zero eigenvectors $v^{\pm}\in \CC^{m}$ such that $\theta_M(v^+)=\zeta v^+$, $\theta_M(v^-)=-\zeta v^-$ and $\operatorname{Im}(\theta_M)$ is generated by $v^{\pm}$. 
	
	Finally, if $\ker(\theta_M)\cap \operatorname{Im}(\theta_M)\not=\{0\}$, then $\theta_M^2$ has rank at most one, which is impossible as $\operatorname{Im}(\theta^2_M)=\operatorname{Im}(\theta_M)$. Hence $\ker(\theta_M)\cap \operatorname{Im}(\theta_M)=\{0\}$ and a dimension counting yields the desired decomposition.
\end{proof}

\subsection{A non-degenerate bilinear map}

\label{ss.Perfect-pairing}

Let $x=[V]\in X$ be an arbitrary point. Let 
\[
\widetilde{T} = (T_{X,[V]}\otimes V) \otimes \sO_{\PP^1} \subset \tilde N
\]
be the trivial factor  and let $T \subset N$ be its pull-back (see Lemma \ref{lem_trivial_factor} and Notation \ref{notationNT}).  For any $t\in \PP^1$, one gets canonical isomorphisms
\[
H^0(C,T) = H^0(C,N) \cong H^0(\PP^1,\widetilde{N}) = H^0(\PP^1,\widetilde{T}) \stackrel{\operatorname{ev}_t}{\cong} T_{X,x}\otimes V.
\]

Let $(F,q)$ be the orthogonal vector bundle associated to $[V]$ 
in Proposition \ref{prop.defineF} and $E\coloneqq F\otimes \alpha$. 
We have seen in Subsection \ref{ss.cM-to-X} that $\dim H^0(C,E)^-=1$.
We come to a key technical point of this paper:

\begin{prop}
	\label{prop.perfect-pairing}
	There exists a well-defined non-degenerate bilinear map
    \begin{equation}
	\label{eq.pefect-pairing}
	\begin{tikzcd}[row sep=tiny]
		\Psi_x\colon 
		H^0(C,N) \times H^0(C,\wedge^2 F\otimes K_C)^+ \arrow[r]
		& H^0(C,E)^- = V \cong \CC \\
		(\sigma,\theta) \arrow[r,mapsto]
		& \theta\circ \sigma.
	\end{tikzcd}
\end{equation}
In particular, we have
	\[
	\dim H^0(C,\wedge^2 F\otimes K_C)^+ = 2g-1.
	\]
\end{prop}

\begin{remarks}
	We remark that an element $\theta\in H^0(C,\wedge^2 F\otimes K_C)$ can be viewed pointwise as a $K_C$-valued skew-symmetric map with respect to $q$.
\end{remarks}

\begin{proof}[Proof of Proposition \ref{prop.perfect-pairing}]

{\em Step 1. We show that $\Psi_x$ is well-defined.}
Let us first explain the meaning of $\theta \circ \sigma$:
We consider the skew-symmetric form $\theta$ as a morphism
$F^* \rightarrow F \otimes K_C$. The orthogonal structure of $F$ gives an isomorphism $E=F\otimes \alpha \cong F^*\otimes \alpha$, so $\theta\circ\sigma$ is defined as the composition:
	\begin{equation}
		\label{eq.def-theta-sigma}
		\sO_C\xrightarrow{\sigma} N \xrightarrow{\iota} E\otimes h^{-(g-1)} \cong F^*\otimes \alpha \otimes h^{-(g-1)} \xrightarrow{\theta} F\otimes K_C\otimes \alpha\otimes h^{-(g-1)} \simeq E
	\end{equation}
    where $\iota$ is the inclusion given \eqref{defineiota}.

    We are left to show that $\theta\circ \sigma$ is an element in $H^0(C,E)^-$. 
    	As $H^0(C,N)\cong H^0(\PP^1,\widetilde{N})$, the section $\sigma$ is $i$-invariant, which implies that $\theta\circ \sigma$ is an element contained in
	\[
	H^0(C,F\otimes K_C\otimes \alpha\otimes h^{-(g-1)})^+ = H^0(C,E\otimes K_C\otimes h^{-(g-1)})^+.
	\]
	On the other hand,  there exists an isomorphism $K_C\cong h^{-(g-1)}$, which changes the sign of the $i$-action of each side, hence $\theta\circ \sigma\in H^0(C,E)^-$.

{\em Step 2. We show that $\Psi_x$ is non-degenerate.}
 	By Lemma \ref{lem.dim-h0-tangent}, we have
	\[
	\dim H^0(C,\wedge^2 F \otimes K_C)^+ \geq 2g-1.
	\] 
	In particular, as $\dim H^0(C,N)=2g-1$ and $\dim H^0(C,E)^-=1$, it is enough to prove that if $\theta\circ \sigma = 0$ for any $\sigma\in H^0(C,N)$, then $\theta=0$. Pick such $\theta\in H^0(C,\wedge^2 F\otimes K_C)^+$. Let $y\in C\setminus W$ be a point. Note that $T_y$ is the image of the evaluation map
	\[
	\operatorname{ev}_y\colon \CC^{2g-1}\cong H^0(C,N) \longrightarrow N_y \cong \CC^{2g}
	\]
	and it is a codimension one subspace of $N_y$. Moreover, as $y\not\in W$, the linear map 
    \begin{equation}
        \label{iotay}
        	\iota_y\colon N_y \longrightarrow (E\otimes h^{-(g-1)})_y
    \end{equation}
	is an isomorphism (cf. the exact sequence \eqref{eq_ext_E}).
	
	On the other hand, as $\theta\circ \sigma=0$ for any $\sigma\in H^0(C,N)$, it follows from \eqref{eq.def-theta-sigma} that the codimension one subspace $\iota_y(T_y)$ of $(E\otimes h^{-(g-1)})_y$ is contained in the kernel of the linear map
	\[
	\theta(y)\colon (E\otimes h^{-(g-1)})_y\longrightarrow E_y.
	\]
	This implies that the rank of $\theta(y)$ is at most one. Nevertheless, since $\theta$ is skew-symmetric, the rank of $\theta(y)$ must be even; so $\theta(y)=0$ for general points $y\in C$ and hence $\theta=0$.
\end{proof}

\subsection{Hitchin morphism of $\cM$}

We can now combine the various results to obtain a suitable description of the cotangent space of $\cM$:

\begin{proposition}
	\label{prop.Cotangent-cM}
	For any point $[\cF]\in \cM$, there exists an $i$-invariant semistable orthogonal vector bundle $(F,q)$ such that $\nu([\cF])=[(F,q)] \in \cU$ and 
	\[
	T^*_{\cM,[\cF]}\cong H^0(C,\wedge^2 F\otimes K_C)^+.
	\]
\end{proposition}

\begin{proof}
 The isomorphism \eqref{morphismMtoX} associates to $[\cF]$ the 
 point $[V]\coloneqq [H^0(C, E)^-] \in X$. 
 Let $[(F,q)] \in \cU$ be the orthogonal bundle associated to $[V]$ in Proposition \ref{prop.defineF}. 
 By Lemma \ref{lem_trivial_factor} the tangent space $T_{X,[V]}$ is canonically isomorphic to $H^0(C,N)\otimes V^*$. The result now follows from 
 $T_{\cM,[\cF]} \simeq T_{X,[V]}$ and  Proposition \ref{prop.perfect-pairing}.
\end{proof}

As a consequence we can introduce a Hitchin morphism $h_{\cM}$ for $\cM$ as follows:

\begin{equation}
	\label{eq.Hitchin-cM}
	h_{\cM}\colon T^*\cM \longrightarrow \cA^+ \coloneqq \bigoplus_{j=1}^{g-1} H^0(C,K_C^{2j})^+ \oplus H^0(C,K_C^g)
\end{equation}
such that for any $[\cF]$ and $\theta\in H^0(C,\wedge^2 F\otimes K_C)^+$, within the isomorphism in Proposition \ref{prop.Cotangent-cM}, we have
\[
h_{\cM}(\theta) = (\operatorname{tr}(\wedge^2\theta),\dots,\operatorname{tr}(\wedge^{2g-2}\theta),\operatorname{Pf}(\theta)),
\]
where $\operatorname{Pf}(\theta)$ is the Pfaffian of $\theta$.

\begin{remark} \label{remark-natural}
Recall that the forgetful map \eqref{define-eta}
$$
\nu^{-}\colon \cM_{\Spin_{2g}}^{-}\longrightarrow \cM_{\SO_{2g}}^{-} \longrightarrow \cM_{\operatorname{O}_{2g}}^{-}
$$
gives a $(\ZZ_2)^{2g+1}$-Galois cover of its image. 
The morphism $\nu: \cM \rightarrow \cU$ is by definition the restriction of $\nu^-$ over the subvariety $\cU \subset \cM_{\operatorname{O}_{2g}}^{-}$,
by Theorem \ref{theorem:ramanan5} it has degree $2^{2g+1}$.
Thus $\cU$ meets the locus where $\nu^-$ is \'etale, and we denote by
$$
\nu_{\acute{e}t}: \cM_{\acute{e}t} \rightarrow \cU_{\acute{e}t}
$$
the restriction.
 Since $\cU$ is (a component of) the fixed locus of $i:\cM_{\operatorname{O}_{2g}} \to \cM_{\operatorname{O}_{2g}}$, we have a splitting 
$$
T^*_{\cM_{\operatorname{O}_{2g}}} \otimes \sO_{\cU}=T^*_{\cU}\oplus (T^*_{\cM_{\operatorname{O}_{2g}}} \otimes \sO_{\cU})^-
$$
into the $i$-invariant factor (which can be identified with $T^*_{\cU}$) and the $i$-anti-invariant factor.
In particular $T^*_{\cU_{\acute{e}t}}$ (and therefore $T^*_{\cM_{\acute{e}t}}$) has a natural structure of submanifold of $T^*_{\cM_{\operatorname{O}_{2g}}}$ (resp. $T^*_{\cM_{\Spin_{2g}}}$).
By Lemma \ref{lemma.invariants} below the formula \eqref{eq.Hitchin-cM} is the extension
of the general definition \eqref{defn:hitchin-general} of the Hitchin morphism for $\cM_{\Spin_{2g}}^{-}$ to this subvariety.
\end{remark}

\begin{lemma} \label{lemma.invariants}
Let $g \geq 2$ be an integer.
	We have the following isomorphisms of algebras of invariants
	\[
	\CC[\mathfrak{spin}_{2g}]^{\Spin_{2g}} \cong \CC[\mathfrak{so}_{2g}]^{\SO_{2g}}.
	\]
    In particular 
    $$
\operatorname{tr}(\wedge^2\bullet),\dots,\operatorname{tr}(\wedge^{2g-2}\bullet),\operatorname{Pf}(\bullet)
$$
is a basis of invariants of $\mathfrak{spin}_{2g}$.
\end{lemma}

\begin{proof}
	Since $\Spin_{2g} \rightarrow \SO_{2g}$ is the universal cover, we have an isomorphism
	$
	\mathfrak{spin}_{2g} \cong \mathfrak{so}_{2g} 
	$
	inducing the isomorphism of invariants. 
    For the second statement just recall from Example \ref{ex.Charac-poly-SO} that 
$$
\operatorname{tr}(\wedge^2\bullet),\dots,\operatorname{tr}(\wedge^{2g-2}\bullet),\operatorname{Pf}(\bullet)
$$
is a basis of invariants of $\fso_{2g}$.
\end{proof}

In the following, we aim to describe the image of $h_{\cM}$. For an element $0 \neq \theta\in H^0(C,\wedge^2 F\otimes K_C)^+$, Proposition \ref{prop.perfect-pairing} gives a surjective linear map
\begin{equation}
	\label{eq.L-theta}
	L_{\theta}\colon \CC^{2g-1}\cong H^0(C,N) \longrightarrow H^0(E)^-\cong \CC, \quad \sigma\longmapsto \theta\circ \sigma.
\end{equation}

\begin{lemma}
	\label{lem.rank-Higgs}
	For any element $\theta\in H^0(C,\wedge^2 F\otimes K_C)^+$ and any point $y\in C$, the rank of $\theta(y)$ is at most two.
\end{lemma}

\begin{proof}
	We may assume $\theta\not=0$, moreover by semi-continuity of the rank function we can assume $y \not\in W$. 
    Let $H \subset H^0(C,N)$ be the kernel of $L_{\theta}$,
    so $\dim H=2g-2$. 
    Since the evaluation map $\operatorname{ev}_y\colon H^0(C,N)\rightarrow N_y$ is injective,
    and $\iota_y$  is an isomorphism (cf. \eqref{defineiota}), we obtain that
    $$
    \iota_y(\operatorname{ev}_y(H)) \subset (E\otimes h^{-(g-1)})_y
    $$
    is a subspace of dimension $2g-2$. Moreover it is contained in the kernel of $\theta(y)$,
    see the construction of $\Psi_x$ in the proof of Proposition \ref{prop.perfect-pairing}.
Thus the rank of $\theta(y)$ is at most two.
\end{proof}

As an immediate consequence we can identify the base of our morphism:

\begin{theorem}
	\label{theorem.Image-h_U}
	In the situation of Setup \ref{setup-odd}, let $h_{\cM}$ be the Hitchin morphism defined by \eqref{eq.Hitchin-cM}. Then $\operatorname{Im}(h_{\cM})\subset H^0(C,K_C^2)^+$.
\end{theorem}

\begin{proof}
	 By Lemma \ref{lem.rank-Higgs}, the rank of $\theta$ at any point $y\in C$ is at most two. This implies $\operatorname{tr}(\wedge^i\theta)=0$ for any $i\geq 3$, from which the statement follows.
\end{proof}

\subsection{Proof of the main theorem \ref{mainthm.Odd-dimension}}

Let 
\[
\Phi_X\colon T^*X\longrightarrow H^0(X,S^2 T_X)^*\cong \CC^{2g-1}
\]
be the morphism given in \cite[Theorem 1.1]{BeauvilleEtesseHoeringLiuVoisin2024}. As $\Phi_X(\lambda v) = \lambda^2 \Phi_X(v)$ for any $v\in T^*X$ and $\lambda\in \CC$, the morphism $\varphi$ induces a rational map 
\[  
\varphi_X\colon \PP T_X \cong \PP (T_X\otimes \sO_X(-1)) \dashrightarrow \PP^{2g-2}.
\]
We need the following geometric interpretation of $\varphi$. Firstly we shall view $\PP^{2g-2}$ as the complete linear system 
\[
|\sO_{\PP^1}(2g-2)| = |2K_{\PP^1} + \Delta |,
\]
and $\PP^1$ parametrizes the pencil of quadrics generated by $q_1$ and $q_2$. Let $x=[V]\in X\subset \PP U^*$ be a general point and let 
\[
H\subset T_{X,x}\otimes V = \left(V^{\perp_{q_1}}\cap V^{\perp_{q_2}}\right)/V
\]
be a general codimension one subspace with $[H]\in \PP (T_X\otimes \sO_X(-1))_x$. Consider the pencil of restricted quadrics $\{q_t|_H \}_{t\in \PP^1}$. Since both $x$ and $H$ are general, there are exactly $(2g-2)$ different degenerate members in this pencil, namely $t_1,\dots,t_{2g-2}$. Let $s_H \in |\sO_{\PP^1}(2g-2)|$ be the unique element
vanishing along the $t_i$'s. Then 
\begin{equation}
    \label{formulaPhi}
   \varphi_X([H]) = s_H
\end{equation}
by \cite[Proposition 3.2]{BeauvilleEtesseHoeringLiuVoisin2024}.

Let now $(F,q)$ be the orthogonal bundle associated to $x=[V] \in X$, and 
$E=F \otimes \alpha$ as in Notation \ref{notation.E}.

For a point $y\in C$ with $t=\pi(y)\in \PP^1$, let $q_t$ be the corresponding quadratic form. Let 
\[
q_E\colon S^2(E\otimes h^{-(g-1)}) \longrightarrow \sO_C(2p_{2g+1})
\]
be the quadratic form induced by $q \colon S^2 F\rightarrow \sO_C$. 
Then its composition with $\iota$ (cf. \eqref{defineiota}) defines a quadratic form
\[
q_E\circ \iota\colon S^2 N \longrightarrow \sO_C(2p_{2g+1}).
\]
By Lemma \ref{lem_trivial_factor} we have a natural inclusion
$$
T_y \coloneqq  (T_{X,x} \otimes V) \otimes \sO_y \subset N_y. 
$$
Recall now from Proposition \ref{prop.defineF} that the quadratic form 
on $F$ is induced by the quadratic form $\tilde q$ and thus 
by $\holom{\bar q}{S^2 (U \otimes \sO_{\PP^1})}{\sO_{\PP^1}(1)}$.
Going through the construction one obtains: 

\begin{lemma}
    \label{lemmaidentifyq}
The restriction of $q_E\circ \iota$ to $T_y$ 
\[
(q_E\circ \iota)_y|_{T_y} \colon S^2 T_y \longrightarrow (\sO_{C}(2p_{2g+1}))_y\cong \CC.
\]
coincides with the restriction $q_t|_{T_{X,x}\otimes V}$.
\end{lemma}
For simplicity of notation we set
$$
q_E|_{T_y} \coloneqq  (q_E\circ \iota)_y|_{T_y}.
$$

\begin{prop}
	\label{prop.degen-theta}
	Given $0 \neq \theta\in H^0(C,\wedge^2 F\otimes K_C)^+$, denote by $H \subsetneq H^0(C,N)$ the kernel of the surjective linear form $L_{\theta}$, and let
    $$
    H_y \subset T_y \subset N_y
    $$
    the image of $H$ under the evaluation map $\operatorname{ev}_y$.
    \begin{enumerate}
		\item If $y\in C\setminus W$ and $q_E|_{H_y}$ is degenerate, then $h_{\cM}(\theta)(y)=0$.
		
		\item Assume in addition that $x\in X\subset \PP^{2g+1}$ is not contained in any coordinate hyperplanes $\{ x_i=0 \}$. Then for any $y\in C$, if $q_E|_{H_y}$ is degenerate, then $h_{\cM}(\theta)(y)=0$.
	\end{enumerate}
\end{prop}

\begin{proof}
Note that $h_{\cM}(\theta)(y)=0$ if and only if all the eigenvalues of $\theta(y)$ are equal to zero and if and only if $\theta(y)$ is nilpotent (see Example \ref{ex.Charac-poly-SO}). Thus we may assume to the contrary that $q_E|_{H_y}$ is degenerate and $\theta(y)$ is non-nilpotent. In particular, the rank of $\theta(y)$ is two by Lemma \ref{lem.rank-Higgs}. 

As $E = F \otimes \alpha$ we can consider $\theta$ as a morphism
$$
	E\otimes h^{-(g-1)} \longrightarrow E.
$$
that is skew-symmetric with respect to $q_E$.

In the first (resp. second) case we know by the exact sequence \eqref{eq_ext_E} (resp. by Lemma \ref{lem.inj-iota}) that the linear map 
	\[
	\iota_y|_{T_y}\colon T_y \longrightarrow (E\otimes h^{-(g-1)})_y
	\]
	is injective. So $\dim \iota_y(H_y)=2g-2$. As $H=\ker(L_{\theta})$, it follows $\theta(y)|_{\iota_y(H_y)}=0$.

	Let $K_y \subset (E\otimes h^{-(g-1)})_y$ be the kernel of $\theta(y)$. 
Lemma \ref{lem.Nilpotent-orth-decomp} gives an orthogonal decomposition
$K_y + \mbox{Im } \theta(y)$, so
bilinear algebra implies 
that the restriction $q_E|_{K_y}$ is non-degenerate. Consequently, since $q_E|_{H_y}$ is degenerate, then $H_y$ must be a proper subspace of $K_y$. In other words, $\dim K_y\geq 2g-1$ and the rank of $\theta(y)$ is at most one, which is a contradiction.
\end{proof}

From now on we assume that $x =[V] \in X\subset \PP^{2g+1}$ is not contained in any coordinate hyperplanes. Modulo the natural $\CC^*$-action, the pairing $\Psi_x$ in Proposition \ref{prop.perfect-pairing} induces a linear isomorphism
\[
\bar{\Psi}_x\colon \PP H^0(\PP^1,\widetilde{N}) = \PP H^0(C,N) \longrightarrow \PP (H^0(C,\wedge^2 F\otimes K_C)^+)^*.
\]
In particular, within the natural isomorphisms
\[
|\sO_{\PP^1}(2g-2)| = | 2 K_{\PP^1} + \Delta | = \PP (H^0(C,K_C^2)^+)^*
\]
one obtains the following diagram:
\begin{equation}
	\label{eq.varphi-h_U}
	\begin{tikzcd}[row sep=large]
		\PP T_{X,x}\cong \PP (T_{X,x}\otimes V) \arrow[dr,"{\varphi_{X}}" below]
		    & \PP H^0(\PP^1,\widetilde{N}) \arrow[r,"{\bar{\Psi}_x}"] \arrow[l,"{\overline{\operatorname{ev}}_t}" above]
		        & \PP (H^0(C,\wedge^2 F\otimes K_C)^+)^* \arrow[dl,"\bar{h}_{\cM,[\cF]}"] \\
		    & \mid \sO_{\PP^1}(2g-2)\mid 
		        & 
	\end{tikzcd} 
\end{equation}
where $t\in \PP^1$ is an arbitrary given point and $\overline{\operatorname{ev}}_t$ is an isomorphism given by the following composition
\[
\operatorname{ev}_t\colon H^0(\PP^1,\widetilde{N}) \longrightarrow \widetilde{T}_t = T_{X,x}\otimes V \subset \widetilde{N}_t.
\]

\begin{prop}
	\label{prop.Commutativity}
	The diagram \eqref{eq.varphi-h_U} is commutative.
\end{prop}

\begin{proof}
By Proposition \ref{prop.Cotangent-cM}, a general element $\theta\in H^0(C,\wedge^2 F\otimes K_C)^+$ determines a point $[\theta]\in \PP T_{\cM,[\cF]}$. Let 
	\[
	H\subset H^0(C,N) = H^0(\PP^1,\widetilde{N})
	\]
	be the kernel of $L_{\theta}$. Then $\bar{\Psi}_x([H])=[\theta]$. As $\Psi_x$ is perfect, we may assume that $H$ is a general codimension one subspace of $H^0(\PP^1,\widetilde{N})$, and hence a general codimension one subspace of $T_{X,x}\otimes V$ under the isomorphism $\overline{\operatorname{ev}}_t$.

    By \eqref{formulaPhi} we have
     $$
     \varphi_X([H])=s_H,
     $$
     where $s_H \in |\sO_{\PP^1}(2g-2)|$ is the unique element
     vanishing in the points $t_1,\dots, t_{2g-2}\in \PP^1$  such that the restricted quadratic form $q_t|_H$ is degenerate if and only if $t=t_i$.
     Since $H$ is general we have $t_i \not\in \Delta$ for all $i=0, \ldots, 2g-1$.
	
	By Lemma \ref{lemmaidentifyq} the restriction $q_t|_H$ can be identified to $q_E|_{H_y}$ with $y=\pi(t)$ under the isomorphism $\operatorname{ev}_y\colon H\rightarrow H_y$. Then Proposition \ref{prop.degen-theta} says that $h_{\cM}(\theta)(y)=0$ for any $y\in C$ such that $\pi(y)=t_i$ for some $i$ and hence 
    $$
    s_H=[h_{\cM}(\theta)]=\bar{h}_{\cM}([\theta]).
    $$
\end{proof}

\begin{proof}[Proof of Theorem \ref{mainthm.Odd-dimension}]
	As $\dim H^0(C,K_C^2)^+=2g-1$, by Theorem \ref{theorem.Image-h_U} and Theorem \ref{thm.BEHLV}, we only need to prove the second statement. By \eqref{eq.varphi-h_U} and Proposition \ref{prop.Commutativity}, the following diagram
	\[
	\begin{tikzcd}[column sep=large,row sep=large]
		\PP T_X = \PP (T_X\otimes \sO_X(-1))\arrow[rr,"{\bar{\Psi}}"] \arrow[dr, "{\varphi_X}" swap]
		&   
		& \PP T_{\cM} \arrow[dl,"\bar{h}_{\cM}"] \\
		& \mid \sO_{\PP^1}(2g-2) \mid=|K_C^2|^+
		&
	\end{tikzcd}
	\]
	commutates over some non-empty open subsets and hence it commutates, which implies the second statement.
\end{proof}

\section{The even dimensional case}

Throughout this section, we let $Y\subset \PP ^{2g}$, $g\geq 2$, be a smooth complete intersection of two quadrics with equations as below:
\[
q'_1\coloneqq \sum_{j=0}^{2g} x_j^2 \quad \text{and} \quad q'_2\coloneqq \sum_{j=0}^{2g} \lambda_j x_j^2,
\]
where $\lambda_j$'s are pairwise distinct numbers. Let $\lambda_{2g+1}$ be a general complex number. Following the notation in \S\,\ref{section:twoquadrics} and \S\,\ref{s.Geometric-inter-hM}, we consider the odd dimensional complete intersection $X\subset \PP^{2g+1}=\PP U^*$ of two quadrics with the defining equations as below:
\[
q_1\coloneqq \sum_{j=0}^{2g+1} x_j^2 \quad \text{and} \quad q_2\coloneqq \sum_{j=0}^{2g+1} \lambda_j x_j^2.
\]

\subsection{Geometry of $Y$ and its cotangent bundle}

Let $\gamma \colon X\rightarrow X$ be the involution by sending $[x_0:\dots,x_{2g}:x_{2g+1}]$ to $[x_0:\dots:x_{2g}:-x_{2g+1}]$. Then $Y$ is exactly the fixed locus of $\gamma$ and $\gamma$ acts on $T_X|_Y$, which yields a splitting
\begin{equation}
\label{eq.tangent-splits}
    T_{X}|_Y = T_Y \oplus N_{Y/X}
\end{equation}
into eigenbundles for the eigenvalues $+1$ and $-1$. 

Fix an arbitrary point $y=[V]\in Y\subset \PP U^*$, where $V\subset U$ is an $1$-dimensional subspace. Let $\ell\subset U$ be the $(2g+1)$-th coordinate line and let $\bar{\ell}$ be its image in the quotient $U/V$, then $\bar{\ell}\not=0$. Recall from Remark \ref{rem.Trivial-factor} that there exists a canonical isomorphism
\[
T_{X,y}\otimes V = T_{X,[V]}\otimes V = \left(V^{\perp_{q_1}} \cap V^{\perp_{q_2}}\right)/V \subset U/V.
\]
It is straightforward to see $N_{Y/X,y}\otimes V=\bar{\ell}$ in \eqref{eq.tangent-splits}, so we obtain:

\begin{fact}
\label{fact.Annihilator}
For every $y =[V] \in Y \subset X$
    the twisted cotangent space $T^*_{Y,y}\otimes V^*$ is canonically isomorphic to the annihilator:
\[     
\bar{\ell}^{\perp}\coloneqq \left\{ w\in \left((V^{\perp_{q_1}}\cap V^{\perp_{q_2}})/V\right)^*  \mid w(\bar{\ell})=0 \right\}.     
\]
\end{fact}

For any $0\leq j\leq 2g+1$, by \cite[Proposition 7.4]{BeauvilleEtesseHoeringLiuVoisin2024}, the quadratic vector field
\[
s_j\coloneq \sum_{k\not = j} \frac{(x_j\partial_k - x_k \partial_j)^2}{\lambda_k-\lambda_j}
\]
in $H^0(X,S^2 T_{\PP^{2g+1}}|_X)$ belong to the image of $H^0(X,S^2 T_X)$ and form a system
of generators for this vector space. Moreover, the splitting \eqref{eq.tangent-splits} yields a natural quotient map
\begin{equation}
\label{eq.Restriction-h-Phi}
    q_Y\colon H^0(X,S^2 T_X) \longrightarrow H^0(Y,S^2 T_Y).
\end{equation}
By \cite[Propositions 7.6 and 7.2]{BeauvilleEtesseHoeringLiuVoisin2024}, the map $q_Y$ is surjective and the kernel is generated by $s_{2g+1}$, which yields:

\begin{fact}
\label{fact.Defining-annihilator}
    The subspace $H^0(Y,S^2 T_Y)^* \subset H^0(X,S^2 T_X)^*$ is the annihilator of $s_{2g+1}$ such that the following diagram commutates:
    \[
    \begin{tikzcd}
        T^*Y \arrow[r] \arrow[d,"{\Phi_Y}" left]
            & T^*X \arrow[d,"{\Phi_X}"] \\
        H^0(Y,S^2 T_Y)^* \arrow[r,"{q^*_Y}"]
            & H^0(X,S^2 T_X)^*
    \end{tikzcd}
    \]
    where the first row is induced by the natural splitting $T_X|_Y = T_Y \oplus N_{Y/X}$.
\end{fact}

\subsection{Proof of Theorem \ref{mainthm.Evendim}}

As explained in \S\,\ref{ss.Quotient-Isom}, the isomorphism $X\cong \cM$ is $\Upsilon$-equivariant. Since we have chosen $\alpha = h^{g-1} \otimes \sO_C(p_{2g+1})$,
the involution $\gamma \colon X\rightarrow X$ corresponds to $\alpha$ by Lemma \ref{lem.PartitionW}. Yet $\alpha$ acts on $\cM$ as the natural involution $i\colon \cM \rightarrow \cM$ by \eqref{eq.Gamma-action-cM}, so we obtain:

\begin{fact}
    There exists an isomorphism $Y\cong \cM^i$.
\end{fact}

In the following we aim to describe the cotangent space of $\cM^i$. Fix an arbitrary point $y=[V]\in Y\subset \PP U^*$. Let $[\cF]\in \cM^i$ be the corresponding $\Spin_{2g}$-bundle and $(F,g)$ the associated orthogonal bundle. 

Recall that $\ell \subset U$ is the $(2g+1)$-th coordinate line.
Then $\bar{\ell}\otimes \sO_{\PP^1}\subset \widetilde{N}$ and we denote by $L$ the pull-back $\pi^* \bar{\ell}\otimes \sO_{\PP^1}$. By \eqref{defineiota} the morphism
\[
\iota \colon N \longrightarrow E\otimes h^{-(g-1)}
\]
is an isomorphism outside $W$. Moreover the restriction over $p_{2g+1}$ has rank $2g-1$ and induces an exact sequence of vector spaces
\begin{equation}
    \label{iota2g+1}
    0\longrightarrow L_{p_{2g+1}} \longrightarrow N_{p_{2g+1}} \stackrel{\iota_{p_{2g+1}}}{\longrightarrow} (E\otimes h^{-(g-1)})_{p_{2g+1}}.
\end{equation}
Let $\bar{L}$ be the saturation of $\iota(L)$ in $E\otimes h^{-(g-1)}$. Moreover, we have:

\begin{lemma}
\label{lem.barL-i-anti-inv}
    There exists an isomorphism $\bar{L}\cong \sO_{C}(p_{2g+1})$ and we have
    \[
    \bar{L}_{p_{2g+1}} = (E\otimes h^{-(g-1)})_{p_{2g+1}}^- = E_{p_{2g+1}}^-\otimes h^{-(g-1)}_{p_{2g-1}}.
    \]
\end{lemma}

\begin{proof}
Since the line bundle $L \simeq \sO_C$ is given by $(2g+1)$-th coordinate line,
the morphism $L \rightarrow \bar L \subset E \otimes h^{-(g-1)}$ 
vanishes exactly in $p_{2g+1}$.
Since $E \otimes h^{-(g-1)}$ has slope one and is semistable by Proposition \ref{prop.defineF}, we have $\deg(\bar{L})\leq 1$ and hence $\bar{L}\cong \sO_C(p_{2g+1})$. Now the second statement follows from the fact that $\bar{L}$ is a line subbunle of $E\otimes h^{-(g-1)}$ and $i$ acts on $\bar{L}_{p_{2g+1}}$ as $-1$.
\end{proof}

The line subbundle $\bar{L}=\sO_C(p_{2g+1})$ of $E\otimes h^{-(g-1)}$ induces a line subbundle $\sO_C$ of $F=E\otimes \alpha$. In particular, Lemma \ref{lem.barL-i-anti-inv} induces a natural identification
\[
\sO_{C}\otimes \sO_{p_{2g+1}} = F^+_{p_{2g+1}.}
\]
Thus the restricted quadratic form $q|_{\sO_C}\colon S^2\sO_C\rightarrow \sO_C$ is non-degenerate as $[\ell]\not\in X$. Let $F'$ be the orthogonal complement of $\sO_C$ in $F$ with respect to $q$. Then we have an orthogonal decomposition
\begin{equation}
    \label{eq.decomp-F}
    F = F'\oplus \sO_C.
\end{equation}

By Fact \ref{fact.Annihilator} and the perfect pairing \eqref{eq.pefect-pairing}, the cotangent space $T^*_{\cM^i,[\cF]}$ is isomorphic to the annihilator of $H^0(C,L)$, i.e.
\begin{equation}
    \label{eq.Cotangent-Mi}
    T_{\cM^i,[\cF]}^*\cong \left\{\theta\in H^0(C,\wedge^2 F\otimes K_C)^+\mid \Psi_y(H^0(C,L),\theta)=0\right\}\eqqcolon L^{\perp}.
\end{equation}

By \eqref{eq.def-theta-sigma} the pairing $\Psi_y(H^0(C,L),\theta)$ is nothing but the restriction of the skew-symmetric map 
\[
\theta\colon F\cong F^*\longrightarrow F\otimes K_C
\]
to the line subbundle $\sO_C$ of $F$. On the other hand, the decomposition \eqref{eq.decomp-F} induces a decomposition
\begin{equation}
\label{eq.Decomp-wedge2F}
    H^0(C,\wedge^2 F\otimes K_C)^+ = H^0(C,\wedge^2 F' \otimes K_C)^+ \oplus H^0(C,F' \otimes \sO_C \otimes K_C)^+.
\end{equation}
Thus the annihilator $L^{\perp}$ is exactly the kernel of the following linear projection 
\begin{equation}
    \label{eq.proj-annil}
    H^0(C,\wedge^2 F\otimes K_C)^+ \longrightarrow H^0(C,F'\otimes K_C)^+,
\end{equation}
which is surjective such that $\dim H^0(C,F'\otimes K_C)^+=1$.

\begin{lemma}
\label{lem.nonvanishing}
    There exists an element $\theta\in H^0(C,\wedge^2 F\otimes K_C)^+$ such that $\theta(p_{2g+1})\not=0$.
\end{lemma}

\begin{proof}
    By the decompositions \eqref{eq.decomp-F} and \eqref{eq.Decomp-wedge2F}, it is enough to show that there exists an element $\theta\in H^0(C,F'\otimes K_C)^+$ such that $\theta(p_{2g+1})\not=0$. Since $F'\otimes K_C$ is $i$-invariant, we consider the following decomposition of vector bundles:
    \[
    \pi_*(F'\otimes K_C) = \pi_*(F'\otimes K_C)^+ \oplus \pi_*(F'\otimes K_C)^-.
    \]
    Then we get
    \[
    \dim H^0(C,F'\otimes K_C)^+ = \dim H^0(\PP^1, \pi_*(F'\otimes K_C)^+) = 1.
    \]
    Since $\pi_*(F'\otimes K_C)^+$ is a vector bundle on $\PP^1$ with a unique section, it has a direct factor $\sO_{\PP^1}$. In particular the section does not vanish.
     Finally note that $i$ acts on $(F'\otimes K_C)_{p_{2g+1}}$ as $+1$, so the natural map
    \[
    \pi^*\left(\pi_*(F'\otimes K_C)^+\right) \longrightarrow F'\otimes K_C
    \]
    is an isomorphism at $p_{2g+1}$ and hence the unique (up to scaling) non-zero element $\theta\in H^0(C,F'\otimes K_C)^+$ does not vanish at $p_{2g+1}$.
\end{proof}

\begin{proposition}
\label{prop.Cotangent-cMi}
    There exists a canonical isomorphism
    \[
    T^*_{\cM^i,[\cF]} \cong H^0(C,\wedge^2 F\otimes K_C\otimes \sO_C(-p_{2g+1}))^+.
    \]
\end{proposition}

\begin{proof}
    By \eqref{eq.Cotangent-Mi}, the cotangent space $T_{\cM^i,[\cF]}^*$ is isomorphic to $L^{\perp}$, which is a codimension one subspace of $H^0(C,\wedge^2 F\otimes K_C)^+$. 
    Thus we have to show that if $\theta\in L^{\perp}$ is a non-zero element, then $\theta$ vanishes at $p_{2g+1}$:
since $\theta \in L^\perp$ the composition
    \[
    \bar{L} \longrightarrow E\otimes h^{-(g-1)} \xrightarrow{\theta} E
    \]
    vanishes identically, thus $\bar{L}_{p_{2g+1}}$ is contained in the kernel of $\theta(p_{2g+1})$. 

    On the other hand, since $y=[V]$ is contained in $\{x_{2g+1}=0\}$, it follows from \eqref{eq.Embedding-H0(C,E)-} that the evaluation map
    \[
    H^0(C,E)^- \longrightarrow E_{p_{2g+1}}^-
    \]
    is zero. In particular, the image $\iota(N_{p_{2g+1}})$ is a $2g-2$ subspace 
    of the kernel of $\theta(p_{2g+1})$.
    
    Recall that $\bar{L}$ is the saturation of $\iota(L)$ in $E\otimes h^{-(g-1)}$.
    Then \eqref{iota2g+1} shows that $\bar{L}_{p_{2g+1}}$ is not contained in  $\iota(N_{p_{2g+1}})$.
    Therefore the kernel of $\theta(p_{2g+1})$ has dimension at least $2g-1$, i.e.
    $\theta(p_{2g+1})$ has rank at most one.
     Yet the rank of $\theta(p_{2g+1})$ is even, so we have $\theta(p_{2g+1})=0$.
\end{proof}

\begin{proof}[Proof of Theorem \ref{mainthm.Evendim}]
Consider the following commutative diagram given by Theorem \ref{mainthm.Odd-dimension}:
\begin{equation}
\label{eq.MainThm-Odd-dim-comm-diag}
    \begin{tikzcd}[row sep=large, column sep=large]
        T^*X \arrow[r,  "\Psi"] \arrow[d,"{\Phi_X}" left] 
            & T^*\cM \arrow[d,"{h_{\cM}}"] \\
        H^0(X,S^2 T_X)^* \arrow[r,"\cong"]
            & H^0(\PP^1,K_{\PP^1}^2\otimes \sO_{\PP^1}(\Delta))
    \end{tikzcd}
\end{equation}
    The splitting \eqref{eq.tangent-splits} and the surjective map \eqref{eq.Restriction-h-Phi} yields a restriction $\Phi_X$ to $T^*Y$ such that
    \[
    \Phi_Y=\Phi_X|_{T^*Y}\colon T^*Y \longrightarrow H^0(Y,S^2 T_Y)^* \subset H^0(X,S^2 T_X)^*.
    \]
    Similarly, it follows from \eqref{eq.Cotangent-Mi} and Proposition \ref{prop.Cotangent-cMi} that the restriction of $h_{\cM}$ to $T^*\cM^i$ gives a map
    \[
    h_{\cM^i} = h_{\cM}|_{T^*\cM^i}\colon T^*\cM^i \longrightarrow H^0(C,K_C^2\otimes \sO_C(-p_{2g+1}))^+,
    \]
    where
    \[
    H^0(C,K_C^2\otimes \sO_C(-p_{2g+1}))^+\cong H^0(\PP^1,K_{\PP^1}^2\otimes \sO_{\PP^1}(\Delta-r_{2g+1}))\cong \CC^{2g-2}.
    \]
We summarise the construction as follows (see Fact \ref{fact.Defining-annihilator}):
    \begin{equation}
    \label{eq.Restriction-inv-locus}
        \adjustbox{scale=0.73,left}{%
\begin{tikzcd}[row sep=large, column sep=large]
    T^*X  \arrow[d,"{\Phi_X}" left] \arrow[rrr,bend left = 20,"\Psi"]
        & T^*Y \arrow[l] \arrow[r,"{\Psi|_{T^*Y}}"] \arrow[d,"{\Phi_Y}" left]
            & T^*\cM^i \arrow[r] \arrow[d,"{h_{\cM^i}}"]
                & T^*\cM \arrow[d,"{h_{\cM}}"] \\
    H^0(X,S^2 T_X)^* \arrow[rrr,bend right = 20, "\cong"]
        & H^0(Y,S^2 T_Y)^* \arrow[l] \arrow[r,"\cong"]
           & H^0(\PP^1,K_{\PP^1}^2\otimes \sO_{\PP^1}(\Delta-r_{2g+1})) \arrow[r]
               & H^0(\PP^1,K_{\PP^1}^2\otimes \sO_{\PP^1}(\Delta))
\end{tikzcd}
}
    \end{equation}
Since \eqref{eq.MainThm-Odd-dim-comm-diag} commutes, so does its restriction
\eqref{eq.Restriction-inv-locus}.
\end{proof}

%\bibliographystyle{alpha}
%\bibliography{biblio}

\newcommand{\etalchar}[1]{$^{#1}$}

\end{document}